\documentclass[preprint,10pt]{amsart}
\usepackage{lineno,hyperref}
\modulolinenumbers[5]

\usepackage{amsfonts,amscd}
\usepackage{amsmath,mathtools,latexsym,amssymb,amsthm}
\usepackage{hyperref}
\usepackage{cite}
\usepackage{lscape,fancyhdr}
\usepackage{epsf,graphicx}
\usepackage{setspace,cite}
\usepackage{lscape,fancyhdr,fancybox}
\usepackage{texdraw}

\makeatletter
\newcommand*\bigcdot{\mathpalette\bigcdot@{.5}}
\newcommand*\bigcdot@[2]{\mathbin{\vcenter{\hbox{\scalebox{#2}{$\m@th#1\bullet$}}}}}
\makeatother

\def\Kmath{\mathbb{K}}

\newcommand{\el}{\end{list}}

\newtheorem{Thm}{Theorem}[section]
\newtheorem{Lem}[Thm]{Lemma}
\newtheorem{Prop}[Thm]{Proposition}
\newtheorem{Def}[Thm]{Definition}
\newtheorem{Exm}[Thm]{Example}
\newtheorem{Rem}[Thm]{Remark}

\newtheorem{Note}[Thm]{Note}

\begin{document}

\title{Deformations of Courant pairs and Poisson algebras}
\author{Ashis Mandal and Satyendra Kumar Mishra}
\footnote{The research of S.K. Mishra is supported by CSIR-SPM fellowship grant 2013.}
\address{Department of Mathematics and Statistics, Indian Institute of Technology,
Kanpur 208016, India.}
\begin{abstract}{

We study deformation of Courant pairs with a commutative algebra base. We consider the deformation cohomology bi-complex and describe a universal infinitesimal deformation. In a sequel, we formulate an extension of a given deformation of a Courant pair to another with extended base. This leads to describe the obstruction in extending a given deformation. We also discuss about the construction of versal deformation of Courant pairs. As an application, we explicitly compute universal infinitesimal deformation of  Poisson algebra structures
  on the three dimensional complex Heisenberg Lie algebra. This provides a comparison of the second deformation cohomology spaces of these Poisson algebra structures by considering them in the  category of Leibniz pairs and Courant pairs, respectively.
}
\end{abstract}
\maketitle
\tableofcontents

\section{Introduction}
We consider the notion of Courant pair as a special type of Courant algebras. Initially, Courant algebras were introduced by H. Bursztyn and coauthors \cite{BCG}, in the context of reduction of Courant algebroids and generalized complex structures. By  definition, a Courant algebra over a Lie algebra $\mathfrak{g}$  is a Leibniz algebra $L$ equipped with a Leibniz algebra homomorphism to the Lie algebra $\mathfrak{g}$. This new type of  algebra is used to interpret the moment map in symplectic geometry as an object which controls an extended part of the action of Courant algebras on Courant algebroids. In fact, any Leibniz algebra $L$ can be viewed as a Courant algebra over the associated Lie algebra $L_{Lie}$. We consider the notion of Courant pair as a Courant algebra over the Lie algebra $Der(A)$, the space of  linear derivations of an associative algebra $A$ (over coefficient field $\mathbb{K}$). 

By definition, a Courant pair denoted by $(A, L)$ consists of an associative algebra $A$ and a Leibniz algebra $L$ over the same coefficient field $\mathbb{K}$, equipped with a Leibniz algebra homomorphism  $\mu :L \rightarrow \mathfrak{g}$, where $\mathfrak{g}$ denotes the Lie algebra of $\mathbb{K}$-linear derivations on $A$  (under the commutator bracket ). If we take the Leibniz algebra in this definition is in particular a Lie algebra then the Courant pair becomes  a Leibniz pair $(A, L)$ -  a generalization of the Poisson algebras, introduced in \cite{LeibP} by M. Flato et al. A formal deformation theory of Leibniz pairs is  studied in \cite{LeibP}. Later it was found as a special case of a homotopy algebra introduced by H. Kajiura and J. Stashef in \cite{KS06}, where they referred it as an open closed homotopy algebra (OCHA). It is evident that Courant pair ( a special type of Courant algebra) naturally appears when one consider a Lie algebroid \cite{LieA} or more generally a Leibniz algebroid, a Loday algebroid, and a Courant algebroid (\cite{FreeN}, \cite{CARoth}, \cite{FreeN}, \cite{BCG}, \cite{KW}, \cite{RW})) over a smooth manifold. In all these cases, the Lie algebras of the associated Courant algebra structures are the Lie algebra of smooth vector fields or space of linear derivations of the associative algebra of smooth functions defined on the manifold. 

In \cite{MM2016}, we defined a deformation complex of a Courant pair and described a one parameter formal deformations of Courant pairs in terms of this deformation complex. Here, we study deformation of Courant pairs with a commutative algebra base. The characterization of all non-equivalent deformations of a given  object is one of the main problems in deformation theory. One way to solve this problem is to construct a versal deformation of the given object, which induces all non-equivalent deformations. The existence of such a versal deformation for algebraic objects, follows by taking few small restrictions on the cohomology space. This versal deformation is unique at the infinitesimal level. For instance, such a construction was first given for Lie algebras in \cite{FiFu}, and  for Leibniz algebra was developed in \cite{FMM}. This motivates for a deformation theory and a construction of versal deformation to follow for more general object namely for Courant pair. We discuss a construction of such a versal deformation, which induces all non-equivalent deformations. Moreover, one can skew-symmetrize this construction to get a  versal deformation for Leibniz pairs \cite{LeibP}. In turn, one gets a versal deformation for Poisson algebras as well by observing it as an object in this more general notion of Courant pair.

 The category of Leibniz pairs is a full subcategory of the category of Courant pairs. We view a Poisson algebra  as a Courant pair and deduce a new deformation complex for Poisson algebras. We consider the classification of Poisson algebra structures on the three dimensional complex Heisenberg Lie algebra appeared in the work of Goze and Remm in \cite{Goze}. Then we compare the second cohomology spaces (of these Poisson algebras) obtained from the deformation complex in \cite{LeibP} and the deduced deformation complex in this paper, respectively. We present an explicit computation of a universal infinitesimal deformation of the Poisson algebra structures as an object in both the categories of Leibniz pairs and Courant pairs. It turns out that, the deformation of Poisson algebras in the category of Courant pairs not only covers all the Leibniz pair deformations but it also gives new deformations which are only Courant pairs.

In the second section, we recall several definitions related to Courant pairs and mention about the deformation cohomology defined in \cite{MM2016}. In the third section, we introduce a deformation of a Courant pair with a commutative algebra base. We construct a universal infinitesimal deformation of a Courant pair. Subsequently, we also construct a versal deformation of a Courant pair. In the last section, we deduce a deformation complex of a Poisson algebra by considering it a Courant pair. As an example, we consider the Poisson algebra structures on three dimensional complex Heisenberg Lie algebra. 

\section{Preliminaries}
Throughout we will consider vector spaces over a field $\mathbb{K}$ of characteristics zero and all maps are $\mathbb{K}$-linear unless otherwise it is specified. 

\subsection{Courant algebras over a Lie algebra}
Now, we recall the definition of  Courant algebras  from \cite{BCG}. Let $\mathfrak{g}$  be a Lie algebra over $\mathbb{K}$.
\begin{Def}\label{Courant algebra}
A {\it Courant algebra }  over a Lie algebra $\mathfrak{g}$  is a Leibniz algebra $L$ equipped with a homomorphism of Leibniz algebras $\pi: L\rightarrow\mathfrak{g}$. We will simply denote it as $\pi: L \longrightarrow \mathfrak{g}$ is a Courant algebra.
\end{Def}

\begin{Exm}
A Courant algebroid over a smooth manifold $M$  gives an example of a Courant algebra over the Lie algebra of vector fields $\mathfrak{g}=\Gamma(TM)$ by taking $\mathfrak{a}$ as the Leibniz algebra structure on the space of sections of the underlying vector bundle of the Courant algebroid.  From \cite{RW} it follows that any Courant algebra is actually an example of a $2$- term $L_\infty$ algebra \cite{BC}.
\end{Exm}

\begin{Def}\label{exact Courant algebra}
An  {\it exact} Courant algebra  over a Lie algebra $\mathfrak{g}$ is a Courant algebra $\pi: L \longrightarrow \mathfrak{g}$ for which $\pi$ is a surjective linear map and $\mathfrak{h}=ker (\pi)$ is abelian, i.e. $[h_1,h_2]=0$ for all $h_1,h_2 \in \mathfrak{h}$.
\end{Def}
For an exact Courant algebra $\pi: L \longrightarrow \mathfrak{g}$, there are two canonical actions of  $\mathfrak{g}$ on $\mathfrak{h}$: 
$[g,h]=[a,h]$ and $[h,g]=[h,a]$ for any $a$ such that $\pi(a)=g$. Thus the $\mathbb{K}$-module $\mathfrak{h}$ equipped with these two actions denoted by the same bracket notation $[-,-]$, is a representation of $\mathfrak{g}$ ( viewed as a Leibniz algebra ). 

The next example will give a natural exact Courant algebra associated with any representation of the Lie algebra $\mathfrak{g}$.

\begin{Exm} As in Example \ref{Lex1}, $\mathfrak{g}$ be a Lie algebra acting on the vector space $V$. Then
$L =\mathfrak{g}\oplus V$ becomes a Courant algebra over $\mathfrak{g}$ with the Leibniz algebra homomorphism given by projection on $\mathfrak{g}$  and the Leibniz bracket on $L$ given by
$$ [(g_1,h_1),(g_2,h_2)]=([g_1,g_2], g_1. h_2),$$
where $g.h$ denotes the action of the Lie algebra $\mathfrak{g}$ on $V$.
\end{Exm}
In \cite {Mandal16} it is shown that exact Courant algebras over a Lie algebra $\mathfrak{g}$ can be characterised via Leibniz $2$- cocycles, and the automorphism group of a given exact Courant algebra is in a one-to-one correspondence with first Leibniz cohomology space of $\mathfrak{g}$. Exact Courant algebras also appeared in the general study of Leibniz algebra extension and a discussion of some unified product for Leibniz algebras is in \cite{AgMi13, Mil15}.

\begin{Exm}

If we consider  a Leibniz representation $\mathfrak{h}$ of the Lie algebra $\mathfrak{g}$, then
$L =\mathfrak{g}\oplus \mathfrak{h}$ becomes a Courant algebra over $\mathfrak{g}$ via the bracket
$$ [(g_1,h_1),(g_2,h_2)]=([g_1,g_2], [g_1, h_2]+[h_1,g_2]),$$
where the actions (left and right) of $\mathfrak{g}$ on $\mathfrak{h}$ are denoted by the same bracket $[-,-]$.
\end{Exm}

\subsection{Courant pairs}
We recall the notion of a Courant pair and some natural examples. We also recall definition of modules over a Courant pair.

Let $A$ be an associative algebra over $\mathbb{K}$ and by $\mathfrak{g}$, we denote the Lie algebra of linear derivations of $A$ with commutator bracket.
\begin{Def}
 A Courant pair  $ (A,L)$ consists of an associative algebra $A$ and a Leibniz algebra $L$  over the same  coefficient ring $\mathbb{K}$, and equipped with  a Leibniz algebra homomorphism
  $\mu: L\longrightarrow \mathfrak{g} $.
\end{Def}
\begin{Rem}\label{RemC}
A Courant pair $(A, L )$ can also be expressed as a triplet $(\alpha,\mu,\lambda)$, where $\alpha : A\times A\rightarrow A$ denotes the associative multiplication map on $A$, $\lambda : L\times L\rightarrow L $ denotes the Leibniz algebra bracket in $L$, and the homomorphism $\mu$ defines an action $\mu: L\times A\rightarrow A$ given by $\mu(x,a)=\mu(x)(a)$ for all $x \in L$ and $a \in A$.
\end{Rem}
Any Leibniz pair is a Courant pair. In particular,
\begin{enumerate}
  \item a Poisson algebra $A$ gives a Courant pair $(A,A)$ with the Leibniz algebra homomorphism $\mu=id$;
  \item any Lie algebroid $(E,[-,-], \rho)$ gives a Courant pair $(C^{\infty}(M),\Gamma E)$ with  the Leibniz algebra homomorphism $\mu = \rho$;. 
  \item any Lie- Rinehart algebra $(A, L)$ is also a Courant pair satisfying the additional conditions that A as a commutative algebra, $L$ as an $A$-module and $\mu: L \rightarrow \mathfrak{g}$ as an $A$-module morphism such that $[x,fy]=f[x,y]+\mu(x)(f)y;$ for $x,y\in L$ and $f\in A$. 
   \end{enumerate}
\begin{Exm}
Let A be an associative commutative algebra and a module over the Lie algebra $\mathfrak{g}=(Der(A),[-,-]_c)$, where  $\mathfrak{g}$ acts on $A$ via derivations. As in  Example (2.4), if we consider the Leibniz algebra $L = Der(A)\oplus A$ , then $(A,Der(A)\oplus A)$ is a Courant pair, where the Leibniz algebra morphism $\mu$ is given by
the projection map from $Der(A)\oplus A$ onto $Der(A)$.
\end{Exm}
\begin{Rem}
In particular, for $A$ to be the space of smooth functions on a smooth manifold this type of examples of Courant pair appears as twisted Dirac structure of order zero in \cite{ACR}.
\end{Rem}
In the sequel, one can find  Courant pairs associated to an algebra by considering modules and  flat connections on the module of the algebra:

Let $A$ be an associative algebra and $V$ be an unitary $A$-module. If $\nabla$ is a flat connection on $V$, then $V$ is a $Der(A)$-module. Therefore, hemisemidirect product of $Der(A)$ and $V$ defines a Leibniz bracket on the A-module $Der(A)\oplus V$ as
  $$[(X,v),(Y,w)]=([X,Y],\nabla_X(w)),$$
where $v,w\in V$. Therefore, we have a Courant pair $(A,Der(A)\oplus V)$ with Leibniz algebra morphism $\mu$ is given by the projection map onto $Der (A) $.   

\begin{Exm}
 Let $(E,[-,-],\rho)$ be a Lie algebroid over a manifold $M$. A flat $E$-connection on a vector bundle $F\rightarrow M$ is also called a representation of E on the vector bundle $F\rightarrow M$ (recall  from \cite{LieA}). Now, $\Gamma E$ is a Lie algebra and $\Gamma F$ is a $\Gamma E$-module (it follows from flatness of $\nabla$). Therefore, by taking hemisemidirect product of $\Gamma E$ and $\Gamma F$, we have a Leibniz algebra bracket on
$\Gamma E \oplus \Gamma F$ given by
$$[(X,\alpha),(Y,\beta)]=([X,Y],\nabla_X(\beta)) $$ 
for $X,Y \in \Gamma E$ and $\alpha,~ \beta\in \Gamma F $. Consider $A =C^{\infty}(M)$, and $L=\Gamma E \oplus \Gamma F$, then $(A,L)$ is a Courant pair with the Leibniz algebra morphism:
$$\mu :\Gamma E \oplus \Gamma F \rightarrow \chi(M)=Der(A)$$ 
where $\mu=\rho \circ \pi_{1}$, and $\pi_{1}$ is the projection map onto $\Gamma E$. 
\end{Exm} 

\begin{Exm}
Suppose $\mathcal{A}$ is a associative and commutative algebra. Let $\Omega^1_{\mathcal{A}}$ be the $\mathcal{A}$-module of K$\ddot{\mbox{a}}$hler differentials \cite{Khlr}, with the universal derivation $$ d_{0} : \mathcal{A}\rightarrow \Omega^1_{\mathcal{A}}.$$ Consider the $\mathcal{A}$-module $Der(A)\oplus \Omega^1_{\mathcal{A}}$ equipped with
the Leibniz bracket  given by $$[(X,\alpha),(Y,\beta)]=([X,Y],L_X\beta-i_Y d_0 \alpha).$$
 Then $(A,Der(A)\oplus \Omega^1_{\mathcal{A}})$ is a Courant pair. This shows that a Courant pair is also appeared in the context of Courant -Dorfman algebras introduced in \cite{CHRtb}.
\end{Exm}

 Let $M$ be an $(A, A)$-bimodule and $P$ be an $L$-module. We denote by $ A+M$ and $L+P$, respectively, the associative and Leibniz semi-direct products. Now we recall the definition of modules over a Courant pair:

\begin{Def}
 A module over a Courant pair $(A,L)$  means a pair $(M,P)$, where $P$ is a Leibniz algebra module over $L$, $M$ is an $(A, A)$-bimodule, and there is a Leibniz algebra homomorphism $\tilde{\mu}:~(L+P) \longrightarrow Der_{\mathbb{K}}(A+M)$ extending the Leibniz algebra morphism $\mu:~L\rightarrow DerA $ in the Courant pair $(A, L)$ in the following sense:
 \begin{enumerate}
 \item $\tilde{\mu}(x,0)(a,0)=\mu(x)(a)$   for  any $x\in L,~a\in A.$\\
 \item $\tilde{\mu}(x,0)(0,m),~\tilde{\mu}(0,\alpha)(a,0)\in M$   for any $x\in L,~ a\in A,~\alpha\in P,~m\in M. $\\
 \item $\tilde{\mu}(0,\alpha)(0,m)=0$ for any $\alpha\in P,~ m\in M.$
 \end{enumerate}
 \end{Def}
 
 If $(A,L)$ is a Courant pair, then the pair $(A, L)$ is a module over the Courant pair $(A,L)$, where $A$ is an $(A,A)$-bimodule by product in $A$, and $L$ is a module over itself by adjoint action. 
 
Also it is important to notice here that for each $\alpha \in P$, the map $A \longrightarrow M$ defined by sending $a$ to $\tilde{\mu}(0,\alpha)(a,0)$ is a derivation of $A$ into $M$, which yields a  map $\phi: P\rightarrow Der(A,M)$.

\subsection{Cohomology bi-complex of a Courant pair}
 Let $(M,P)$ be a module over the Courant pair $(A,L)$, then $M$ is an $(A,A)$-bimodule and $P$ is a Leibniz-$L$-module. Here we denote the tensor modules $\otimes^{p}A$ simply by $A^p$.

 Let $C^p:=C^p(A,M)=Hom(A^p,M)$, be the vector space of $p$-th Hochschild cochain of $A$ with coefficients in $M$. Note that $C^p$ becomes an $L$-module (symmetric Leibniz algebra module), where the action $[-,-]:L\times C^p\rightarrow C^p$ is given as follows:
  $$[x,f](a_1,\cdots,a_p)=[x,f(a_1,\cdots,a_p)]-\sum_{i=1}^p f(a_1,\cdots,[x,a_i],\cdots,a_p)$$
 and $[f,x]=-[x,f]$ for $f\in C^p$, $a_1, \cdots,a_p \in A$ and $x\in L $.
Here, $[x,m]=-[m,x]=\tilde{\mu}(x)(m)$ for$~a\in A,~m\in M,$ and $~x\in L $.
 
 It follows that the Lie algebra of linear derivation from $A$ to $M$ is a Leibniz algebra submodule of $C^1$ and the map $\phi: P\rightarrow Der(A,M)$ is a Leibniz algebra module homomorphism.
In other words,
 $[x,f]\in Der(A,M)$ for $x\in L$ and $f\in Der(A,M).$
 Also, $\phi[x,p]=[x,\phi(p)]$ (using the Leibniz algebra structure on the semi-direct product space $L+P$).
 
Let $\delta_H$ be the Hochschild coboundary of $A$ with coefficients in the bimodule $M$. Then it follows that $\delta_H$ is a Leibniz algebra module homomorphism, i.e. $$\delta_H([x,f])=[x,\delta_Hf].$$
for all $x\in L$ and $f\in C^p$.

Let us recall the definition of a bicomplex of the Courant pair $(A,L)$ with coefficients in the module $(M,P)$: Define $C^{p,q}(A,L;M,P):=Hom_{\mathbb{K}}(L^q,C^p(A,M))\cong Hom_{\mathbb{K}}(L^q\otimes A^p,M)$ for all  $(p,q)\in \mathbb{Z}\times \mathbb{Z}$ with $p>0$, q$\geq$0, and $C^{0,q}(A,L;M,P):=Hom_{\mathbb{K}}(L^q,P)$. Next, the vertical and horizontal maps in the bicomplex are given as follows:
\begin{itemize}
\item For $p>0$, the vertical coboundary map $C^{p,q}\rightarrow C^{p+1,q}$ is the Hochschild coboundary map $\delta_H$, given by \eqref{Hochschild Coboundary}.
\item For $p=0$ and for any $q\geq 0$, the map $\delta_v:C^{0,q}\rightarrow C^{1,q}$ is the induced map by the action of $P$ on $A$ by derivations. 
\item For $p\geq 0$ and $q\geq 0$, the horizontal map $C^{p,q}\rightarrow C^{p,q+1}$ is the Leibniz algebra coboundary $\delta_L$, given by \eqref{Leibniz coboundary}. 
\end{itemize}

Note that the composition of any two vertical maps is zero since $\delta_H\circ \delta_H=0$, and $Der(A,M)$ is a subspace of $Ker(\delta_H)$. Moreover, any vertical map is a Leibniz module homomorphism, it follows that they commute with the horizontal maps. Thus, we have a total cochain complex given as follows: for $n\geq 0$,
 $$C^n_{tot}( A,L )=\bigoplus _{p+q=n}C^{p,q}( A,L )~\mbox{and}~\delta_{tot}:C^n_{tot}( A,L )\rightarrow C^{n+1}_{tot}( A,L ),$$ whose restriction to $C^{p,q}( A,L )$ by definition is $\delta_H+(-1)^p\delta_L$. Next, define the cohomology of the Courant pair $(A,L)$ with coefficients in the module $(M,P)$ as the cohomology of this total complex. Denoted this cohomology by $H_{CP}^* (A,L; M,P)$. 
 
 In particular for $(M,P)=(A,L)$, denote the cohomology $H_{CP}^* (A,L; A,L)$ simply by $H_{CP}^*(A,L)$. In fact, this is the deformation cohomology of the Courant pair $(A,L)$.
\section{Deformation of Courant pairs}
Let $\mathbb{K}$ be a field of characteristic 0, and $R$ be a commutative algebra with identity over $\mathbb{K}$.
Let $\epsilon: R\rightarrow \mathbb{K}$ be a fixed augmentation, that is an algebra homomorphism with $\epsilon(1)=1$ and $ker(\epsilon)=\mathcal{M}$. We also assume that $dim(\mathcal{M}^k/\mathcal{M}^{k+1})<\infty $. 
\begin{Def}
A deformation $\lambda$ of a Courant pair $(A,L)$ 
 is a Courant pair structure on $(R\otimes A, R\otimes L)$ given by a triplet $\big(\cdot_{\lambda},\mu_{\lambda},[-,-]_{\lambda}\big)$, where 
\begin{itemize}
\item the map $\cdot_{\lambda}:(R\otimes A)\otimes_R (R\otimes A)\rightarrow (R\otimes A)$ gives an associative $R$-algebra structure on $R\otimes A$,
\item the bracket $[-,-]_{\lambda}:(R\otimes L)\otimes_R (R\otimes L)\rightarrow (R\otimes L)$ gives a $R$-Leibniz algebra structure on $R\otimes L$, and 
\item $R$-Leibniz algebra homomorphism $\mu_{\lambda}:R\otimes L\rightarrow Der_R(R\otimes A)$ defines an action of $R\otimes L$ on $R\otimes A$, 
\end{itemize} 
such that the pair of $R$-linear maps: 
$$(\epsilon\otimes id_A,\epsilon\otimes id_L): (R \otimes A, R\otimes L) \rightarrow (\mathbb{K}\otimes A,\mathbb{K}\otimes L)$$
is a Courant pair homomorphism.
\end{Def}

A deformation of Courant pair $(A,L)$ is called {\bf local} if 
$R$ is a local algebra over $\mathbb{K}$, and is called {\bf infinitesimal} if in addition $\mathcal{M}^2=0$.
\begin{Note}
 For $r,s \in R, ~a,b\in A, $ and $x,y\in L$ we have 
$$[r \otimes x, s\otimes y]_{\lambda}=rs[1\otimes x,1\otimes y]_{\lambda},$$
$$(r\otimes a)._{\lambda}(s\otimes b)=rs(1\otimes a)._{\lambda}(1\otimes b),$$
and $$\mu_{\lambda}(r\otimes x)(s\otimes a)=rs\mu_{\lambda}(1\otimes x)(1\otimes a).$$

Therefore, to define a deformation $\lambda$ of a Courant pair, it is sufficient to define the values of $[1\otimes x,1\otimes y]_{\lambda}, ~(1\otimes a)._{\lambda}(1\otimes b)$, and $\mu_{\lambda}(1\otimes x)(1\otimes a)$ for $a,b\in A, $ and $x,y\in L$. Since $$(\epsilon\otimes id_A,\epsilon\otimes id_L):(R \otimes A, R\otimes L) \rightarrow (\mathbb{K}\otimes A,\mathbb{K}\otimes L)$$ is a Courant pair homomorphism, let us observe the following:  
$$[1\otimes x,1\otimes y]_{\lambda}-1\otimes[x,y]\in ker(\epsilon\otimes id_L),$$
$$(1\otimes a)._{\lambda}(1\otimes b)-1\otimes a.b\in ker(\epsilon\otimes id_A) ~~\mbox{and}$$
$$\mu_{\lambda}(1\otimes x)(1\otimes a)-1\otimes \mu(x)(a)\in ker(\epsilon\otimes id_A).$$

Hence, we obtain the following equations:
\begin{align}\label{product}
(1\otimes a)._{\lambda}(1\otimes b)&=1\otimes ab+\sum_j m_j\otimes a_j,\\\label{action}
\mu_{\lambda}(1\otimes x)(1\otimes a)&=1\otimes \mu(x)(a)+\sum_k n_k\otimes b_k,\\\label{bracket}
[1\otimes x,1\otimes y]_{\lambda}&=1\otimes[x,y]+\sum_i p_i\otimes x_i,
\end{align}
where $m_j,n_k,p_i\in \mathcal{M}:=Ker(\epsilon)$, $a_j,b_k\in A$ and $x_i,\in L$ for $i,j$ and $k$ varies over a finite index set.
\end{Note}
\begin{Def}
We say two deformation $\lambda_1,$ and $\lambda_2$ are equivalent if there exists a Courant pair isomorphism
$$(\varphi_A,\varphi_L): (R\otimes A,R\otimes L)_{\lambda_1}\rightarrow  (R\otimes A,R\otimes L)_{\lambda_2}$$
such that $(\epsilon\otimes id_A)\circ \varphi_A =(\epsilon\otimes id_A)$, and $(\epsilon\otimes id_L)\circ \varphi_L =(\epsilon\otimes id_L).$
\end{Def}

\begin{Def}
Let $\lambda$ be a given deformation of Courant pair $(A,L)$ with base $(R,\mathcal{M})$ and augmentation $\epsilon: R\rightarrow \mathbb{K}$. Let $R^{\prime}$ be another commutative algebra with identity and a fixed augmentation  $\epsilon^{\prime}: R^{\prime}\rightarrow \mathbb{K}$. If $\phi : R\rightarrow R^{\prime}$ is an algebra homomorphism with $\phi(1)=1$, $\epsilon^{\prime}\circ \phi= \epsilon $, and $\mathcal{M}^{\prime}:=ker(\epsilon^{\prime})$, then the push-out $\phi_{*}\lambda$ is a deformation of $(A,L)$ with base $(A^{\prime},\mathcal{M}^{\prime})$, where
$$(r^{\prime} \otimes_R (r\otimes a))._{\phi_{*}\lambda}(s^{\prime}\otimes_R(s\otimes b))=r^{\prime}s^{\prime}\otimes_R (r\otimes a)._{\lambda}(s\otimes b);$$
$$\mu_{\phi_{*}\lambda}(r^{\prime} \otimes_R (r\otimes x)) (s^{\prime}\otimes_R (s\otimes a))=r^{\prime}s^{\prime}\otimes_R\mu_{\lambda}(r\otimes x)(s\otimes a),$$
$$[r^{\prime} \otimes_R (r\otimes x), s^{\prime}\otimes_R (s\otimes y)]_{\phi_{*}\lambda}=r^{\prime}s^{\prime}\otimes_R[r\otimes x,s\otimes y]_{\lambda},$$
for all $a,b\in A,~r,s\in R,~r^{\prime},s^{\prime}\in R^{\prime}$. Here, $R^{\prime}$ is considered as an $R$-module by the action $r^{\prime}.r=r^{\prime}\phi(r)$, so $R^{\prime}\otimes L= R^{\prime}\otimes_R (R \otimes L),$ and $R^{\prime}\otimes A= R^{\prime}\otimes_R (R \otimes A).$
\end{Def}

Note that if $\lambda:=\big(\bigcdot_{\lambda},\mu_{\lambda},[-,-]_{\lambda}\big)$ is given by equations \eqref{product}-\eqref{bracket}, then we can write the push-out $\phi_*\lambda:=\big(\bigcdot_{\phi_{*}\lambda},\mu_{\phi_{*}\lambda},[-,-]_{\phi_{*}\lambda}\big)$ as follows:
\begin{align}
(1\otimes a)\bigcdot_{\phi_{*}\lambda}(1\otimes b)&=1\otimes ab+\sum_j \phi(m_j)\otimes a_j,\\
\mu_{\phi_{*}\lambda}(1\otimes x)(1\otimes a)&=1\otimes \mu(x)(a)+\sum_k \phi(n_k)\otimes b_k,\\
[1\otimes x,1\otimes y]_{\phi_{*}\lambda}&=1\otimes[x,y]+\sum_i \phi(p_i)\otimes x_i,
\end{align}
where $m_j,n_k,p_i\in \mathcal{M}:=Ker(\epsilon)$, $a_j,b_k\in A$ and $x_i,\in L$ for $i,j$ and $k$ varies over a finite index set.

\subsection{Construction of a Universal Infinitesimal Deformation}
Let $(A,L)$ be a Courant pair, which satisfies the condition $dim(H^2_{CP}(A,L))<\infty$. Let us denote $\mathcal{H}_{CP}:=H^2_{CP}(A,L)$. Define a $\mathbb{K}$-algebra $$C_1 := \mathbb{K}\oplus\mathcal{H}_{CP}^{\prime}$$
with multiplication: $(k_1,h_1)\cdot(k_2,h_2)=(k_1k_2,k_1h_2+k_2h_1)$. Note that $\mathcal{H}_{CP}^{\prime}$ is an ideal of $C_1$ with zero multiplication. Define a $\mathbb{K}$-linear map 
$$\nu:\mathcal{H}_{CP}\rightarrow Hom(A^2,A)\oplus Hom(L,Hom(A,A))\oplus Hom(L^2,L).$$
by sending a cohomology class to a cocycle representing it. Let us observe the following isomorphisms of $\mathbb{K}$ vector spaces:
\begin{itemize}
\item $C_1\otimes A\cong A\oplus Hom(\mathcal{H}_{CP};A),$ and
\item $C_1\otimes L\cong A\oplus Hom(\mathcal{H}_{CP};L),$
\end{itemize} 
Using the above identifications, let us define the following: 
\begin{itemize}
\item A multiplication $\bigcdot_{\eta_1}:(C_1\otimes A)\times (C_1\otimes A)\rightarrow (C_1\otimes A)$ given by 
\begin{equation*}
(a,\phi_1)\bigcdot_{\eta_1}(b,\phi_2)=(a.b,\psi),
\end{equation*}
where the map $\psi:\mathcal{H}_{CP}\rightarrow A$ is defined as 
\begin{equation*}
\psi(\alpha)=\nu^1(\alpha)(a,b)+\phi_1(\alpha)\cdot b+a\cdot\phi_2(\alpha)
\end{equation*}
for any $\alpha\in \mathcal{H}_{CP}$. 

\item An action $\mu_{\eta_1}:(C_1\otimes L)\times (C_1\otimes A)\rightarrow C_1\otimes A$ given by: 
\begin{equation*}
\mu_{\eta_1}((x,\psi),(a,\phi))=(\mu(x)(a),\varphi),
\end{equation*}
where the map $\varphi:\mathcal{H}_{CP}\rightarrow A$ is defined as  
\begin{equation*}
\varphi(\alpha)=\nu^2(\alpha)(x,a)+\mu(\psi(\alpha))(a)+\mu(x)(\phi(a))
\end{equation*}
for any $\alpha\in \mathcal{H}_{CP}$. 

\item A bracket $[-,-]_{\eta_1}:(C_1\otimes L)\times (C_1\otimes L)\rightarrow(C_1\otimes L)$ given by
\begin{equation*}
[(x,\psi_1),(y,\psi_2)]_{\eta_1}=([x,y],\phi)
\end{equation*}
where the map $\phi:\mathcal{H}_{CP}\rightarrow L$ is defined as  
\begin{equation*}
\phi(\alpha)=\nu^3(\alpha)(x,y)+[\psi_1(\alpha),y]+[x,\psi_2(\alpha)]
\end{equation*}
for any $\alpha\in \mathcal{H}_{CP}$. 
\end{itemize}

By using the condition $\delta_{tot}(\nu(\alpha))=0$, it follows that the triplet $\eta_1:=\big(~\bigcdot_{\eta_1},\mu_{\eta_1},[-,-]_{\eta_1}\big)$ gives a Courant pair structure on $(C_1\otimes A,C_1\otimes L)$ $($over the algebra $C_1)$. Therefore, we have an infinitesimal deformation $\eta_1$ of the Courant pair $(A,L)$ with base $(C_1,\mathcal{H}_{CP})$.
\begin{Rem}
Let $\tilde{\nu}$ be another map associating a cohomology class in $\mathcal{H}_{CP}$ to a cocycle representing this class. Then the cocycles $\nu(\alpha)$ and $\tilde{\nu}(\alpha)$ in $C^2_{tot}(A,L)$ represent the same class $\alpha\in \mathcal{H}_{CP}$. Let $\{h_i\}_{1\leq i\leq n}$ be a basis of $\mathcal{H}_{CP}$, then define a $K$-linear map
$$\gamma: \mathcal{H}_{CP}\rightarrow C^1_{tot}(A,L)=Hom(A,A)\oplus Hom(L,L)$$
by $\gamma(h_i)=(\gamma^1(h_i),\gamma^2(h_i))=(\gamma^1_i,\gamma^2_i)$, where $\gamma^1_i\in Hom(A,A)$, and $\gamma^2_i\in Hom(L,L)$ such that 
\begin{equation}\label{rel1}
\delta_{tot}(\gamma^1_i,\gamma^2_i)=\nu(h_i)-\tilde{\nu}(h_i)=(\nu^1(h_i)-\tilde{\nu}^1(h_i),\nu^2(h_i)-\tilde{\nu}^2(h_i),\nu^3(h_i)-\tilde{\nu}^3(h_i)),
\end{equation} 
i.e. $\delta_{tot}(\gamma)=\nu-\tilde{\nu}.$ Next, let us define the following maps:
\begin{enumerate}
\item $\Phi_A: C_1\otimes A\rightarrow C_1\otimes A$ by $\Phi_A(a,\theta)=(a,\eta)$, where $\eta(\alpha)=\theta(\alpha)+\gamma^1(\alpha)(a)$ for any $\alpha\in \mathcal{H}_{CP}$, and 
\item $\Phi_L: C_1\otimes L\rightarrow C_1\otimes L$ by $\Phi_L(l,\phi)=(a,\psi)$, where $\psi(\alpha)=\phi(\alpha)+\gamma^2(\alpha)(l)$ for any $\alpha\in \mathcal{H}_{CP}$.
\end{enumerate}

Next, by using equation \eqref{rel1}, it follows that the map $(\Phi_A,\Phi_L):(C_1\otimes A,C_1\otimes L)\rightarrow (C_1 \otimes A,C_1\otimes L)$ is a $C_1$-linear automorphism of the Courant pair $(C_1\otimes A,C_1\otimes L)$. Therefore, the above infinitesimal deformation is independent of the choice of the map $\nu$. 
\end{Rem}
 Let $\lambda$ be an infinitesimal deformation of A Courant pair $(A,L)$ with a finite dimensional base $(R,\mathcal{M})$. Let $\{x_i\}_{1\leq i\leq n}$ be a basis of $\mathcal{M}=ker(\epsilon)$ and $\{\chi_i\}_{1\leq i\leq n}$ be the dual basis. Also, any element $\xi\in \mathcal{M}^
{\prime}$ can be viewed as an element in the dual space $R^{\prime}$ with $\xi(1)=0$. For any such $\xi$ set 
\begin{align*}
\alpha^1_{\lambda,\xi}(a,b)&=\xi\otimes id_A\big((1\otimes a)._\lambda)(1\otimes b)\big),\\
\alpha^2_{\lambda,\xi}(x,a)&=\xi\otimes id_A(\mu_{\lambda}(1\otimes x)(1\otimes a)),\\
\alpha^3_{\lambda,\xi}(x,y)&=\xi\otimes id_L([1\otimes x,1\otimes y]_{\lambda}).
\end{align*}
The triplet $\alpha_{\lambda,\xi}:=(\alpha^1_{\lambda,\xi},\alpha^2_{\lambda,\xi},\alpha^3_{\lambda,\xi})$ is a $2$-cochain in $C^2_{tot}(A,A).$ Moreover, by using equations \eqref{product}-\eqref{bracket} we can write down the Leibniz bracket $[-,-]_{\lambda}$ on $R\otimes L$, associative product $\bigcdot_{\lambda}$ on $R\otimes A$, and the action $\mu_{\lambda}$ of $R\otimes L$ on $R\otimes A$ in terms of the basis of $\mathcal{M}$, as follows:
\begin{align}\label{asso}
(1\otimes a)\bigcdot_{\lambda}(1\otimes b)&=1\otimes ab+\sum_i m_i\otimes \alpha^1_{\lambda,\xi_i}(a,b),\\\label{act}
\mu_{\lambda}(1\otimes x)(1\otimes a)&=1\otimes\mu(x)(a)+\sum_i m_i\otimes \alpha^2_{\lambda,\xi_i}(x,a),\\\label{brac}
[1\otimes x,1\otimes y]_{\lambda}&=1\otimes[x,y]+\sum_i m_i\otimes \alpha^3_{\lambda,\xi_i}(x,y).
\end{align}

\begin{Lem}\label{cocycle}
The cochain $\alpha_{\lambda,\xi}$ is a cocycle.
\end{Lem}
\begin{proof}
By using equations \eqref{asso}-\eqref{brac}, and using the fact that $(\bigcdot_{\lambda},\mu_{\lambda},[-,-]_{\lambda})$ is a Courant pair structure on the pair $(R\otimes A,R\otimes L)$, we get $\delta_{tot}(\alpha_{\lambda,\xi})=0$, i.e. we have the following equations: $\delta_H \alpha^1_{\lambda,\xi}=0$; $\delta_H \alpha^2_{\lambda,\xi}+\delta_L\alpha^1_{\lambda,\xi}=0$; $ \alpha^3_{\lambda,\xi}-\delta_L\alpha^2_{\lambda,\xi}=0$; and $
\delta_L \alpha^3_{\lambda,\xi}=0$.
\end{proof}

\begin{Thm}
If $\lambda_1$ and $\lambda_2$ are infinitesimal deformations of the Courant pair $(A,L)$ with the base $(R,\mathcal{M})$, then both of them are equivalent if and only if $\alpha_{\lambda_1,\xi}$ and $\alpha_{\lambda_2,\xi}$ represent the same cohomology class for $\xi\in \mathcal{M}^{\prime}$.
\end{Thm}

\begin{proof}
Let $\lambda_1$ and $\lambda_2$ are equivalent infinitesimal deformations of the Courant pair $(A,L)$ with the base $(R,\mathcal{M})$, which implies there exists an $R$-linear isomorphism of Courant pairs, given by
$$(\Phi_A,\Phi_L):(R\otimes A,R\otimes L)_{\lambda_1}\rightarrow(R\otimes A,R\otimes L)_{\lambda_2}.$$
such that $\big((\epsilon\otimes id_A)\circ \Phi_A,(\epsilon\otimes id_L)\circ \Phi_L\big)=(\epsilon\otimes id_A,\epsilon\otimes id_L)$, where $\epsilon:R\rightarrow \mathbb{K}$ is an augmentation such that $Ker(\epsilon)=\mathcal{M}$. Since $R=\mathbb{K}\oplus \mathcal{M}$, we get $(R\otimes A,R\otimes L)=(A\oplus (\mathcal{M}\otimes A),L\oplus (\mathcal{M}\otimes L))$. Let $\{x_i\}_{1\leq i\leq n}$ be a basis of $\mathcal{M}$ and $\{\chi_i\}_{1\leq i\leq n}$ be the dual basis of $\{x_i\}_{1\leq i\leq n}$.

Now, maps $\Phi_A$, and $\Phi_L$ are $R$-linear maps, i.e. $\Phi_A$ is identified with values: $\Phi_A(1\otimes a)$ for $a\in A$. Similarly, $\Phi_L$ is identified with values: $\Phi_L(1\otimes l)$ for $l\in L$. Therefore, the isomorphism $(\Phi_A,\Phi_L)$ is of the form: $(id_A+\Phi_A^2,id_L+\Phi_L^2)$, such that 
$$\Phi_L(1\otimes l)=1\otimes l+\sum_{i=1}^n x_i\otimes \phi^L_i(l),$$
and 
$$\Phi_A(1\otimes a)=1\otimes a+\sum_{i=j}^n x_j\otimes \phi^A_j(a),$$ 
where we identify $\Phi_A^2\in Hom(A,\mathcal{M}\otimes A)$ with $\sum_{i=1}^n x_i\otimes \phi^A_i\in \mathcal{M}\otimes Hom(L,L)$ and similarly $\Phi_L^2\in Hom(L,\mathcal{M}\otimes L)$ with $\sum_{i=1}^n x_i\otimes \phi^L_i\in \mathcal{M}\otimes Hom(L,L)$.

 Then similar to equations \eqref{asso}-\eqref{brac}, the triplets $(~\bigcdot_{\lambda_k},\mu_{\lambda_k},[-,-]_{\lambda_k})$ can be expressed in the form of the basis $\{x_i\}_{1\leq i\leq n}$ and the set $\{\alpha_{\lambda_k,\chi_i}:1\leq i\leq n\}$. Next, by a straightforward calculation it follows that $(\Phi_A,\Phi_L)$ is an isomorphism of Courant pairs if and only if $\delta_{tot}(\phi^A_i,\phi^L_i)=\alpha_{\lambda_1,\chi_i}-\alpha_{\lambda_2,\chi_i}$.
\end{proof}

\begin{Thm}
For any infinitesimal deformation $\lambda$ of a Courant pair $(A,L)$ with a finite base $(R,\mathcal{M})$, there exists a unique homomorphism $\phi:C_1\rightarrow R$ such that $\lambda$ is equivalent to the push-out $\phi_{*}\eta_1$. 
\end{Thm}
\begin{proof}
Let $\lambda$ be an infinitesimal deformation of a Courant pair $(A,L)$ with the base $(R,\mathcal{M})$, where $R$ is finite dimensional local algebra over $\mathbb{K}$. 
Let                                                                                                                                                                                                                                                                                                                                                                                                                                                                                                                                                                                                                                                                                                                                                                                                                                                                                                                                                                                                                                                                                                                                                                                                                                                                                                                                                                                                                                                                                                                                                                                                                                                                                                                                                                                                                                                                                                                                                                                                    $\{x_i\}_{1\leq i\leq n}$ be a basis of $\mathcal{M}$ and $\{\chi_i\}_{1\leq i\leq n}$ be the corresponding dual basis of $\mathcal{M}^{\prime}.$ From Lemma \ref{cocycle}, we define a linear map $$\alpha_\lambda : \mathcal{M}^{\prime} \rightarrow C^2_{tot}(A,L)$$
sending $\xi\mapsto \alpha_{\lambda,\xi}$ for any $\xi\in \mathcal{M}^{\prime}$ and $\delta_{tot}\circ \alpha_\lambda=0$. For $\chi_i\in \mathcal{M}^{\prime}$, let $a_{\lambda, \chi_i}\in \mathcal{H}_{CP}$ be the cohomology class of the cocycle $\alpha_{\lambda, \chi_i}.$ Then, we get a linear map $a_\lambda : \mathcal{M}^{\prime}\rightarrow \mathcal{H}_{CP}$ given by $\xi\mapsto [\alpha_{\lambda,\xi}]$.

Let $a_{\lambda}^{\prime}: \mathcal{H}_{CP}^{\prime}\rightarrow \mathcal{M}$ be the dual map of the map $a_{\lambda}:\mathcal{M}^{\prime}\rightarrow \mathcal{H}_{CP}$, then define 
$$\phi:=(id\oplus a_{\lambda}^{\prime}):C_1\rightarrow  R=\mathbb{K}\oplus\mathcal{M}.$$ Also assume that $\{h_i\}_{1\leq i\leq m}$ is a basis of $\mathcal{H}_{CP}$ and $\{g_i\}_{1\leq i\leq m}$ is the corresponding dual basis of $\mathcal{H}_{CP}^{\prime}$. Then we can write the triplet $(~\bigcdot_{\phi_{*}\beta},\mu_{\phi_{*}\beta},[-,-]_{\phi_{*}\beta})$ in terms of this basis of $\mathcal{H}_{CP}$ as follows:

\begin{align*}
(1\otimes a)\bigcdot_{\phi_{*}\beta}(1\otimes b)&=1\otimes ab+\sum_i \phi(g_i)\otimes \nu^1(h_i)(a,b),\\
\mu_{\phi_{*}\beta}(1\otimes x)(1\otimes a) &=1\otimes\mu(x)(a)+\sum_i \phi(g_i)\otimes \nu^2(h_i)(x,a),\\
[1\otimes x,1\otimes y]_{\phi_{*}\beta}&=1\otimes[x,y]+\sum_i \phi(g_i)\otimes \nu^3(h_i)(x,y).
\end{align*}

Any basis element $\chi_i\in \mathcal{M}^{\prime}$ can be extended to an element of $R^{\prime}$ such that $\chi_i(1)=0$. Also, recall the map $$\nu: \mathcal{H}_{CP}\rightarrow Hom(A^2,A)\oplus Hom(L,Hom(A,A))\oplus Hom(L^2,L)$$ given by $\nu(\alpha)=(\nu^1(\alpha),\nu^2(\alpha),\nu^3(\alpha))$ for any $\alpha\in \mathcal{H}_{CP}$. Then 

\begin{align*}
\alpha_{\phi_*\eta_1,\chi_i}^2(x,a)
&=(\chi_i\otimes id)(\mu_{\phi_{*}\eta_1}(1\otimes x)(1\otimes a))\\
&=(\chi_i\otimes id)\big(1\otimes \mu(x)(a)+\sum_k \phi(g_k)\otimes \nu^2(h_k)(a,b)\big)\\
&=(\chi_i\otimes id)(\sum_k a_{\lambda}^{\prime}(g_k)\otimes \nu^2(h_k)(x,a))\\
&=\sum_k \chi_i(a_{\lambda}^{\prime}(g_k))\otimes \nu^2(h_k)(x,a)\\
&=\sum_k g_k(a_{\lambda}(\chi_i))\otimes \nu^2(h_k)(x,a)\\
&=\nu^2\Big(\sum_k g_k(a_{\lambda}(\chi_i))h_k\Big)(x,a)=\nu^2\circ a_{\lambda}(\chi_i)(x,a).
\end{align*}

Similarly, $$\alpha_{\phi_*\eta_1,\chi_i}^1(a,b)=\nu^1\circ a_{\lambda}(\chi_i)(a,b),$$ 

and $$\alpha_{\phi_*\eta_1,\chi_i}^3(a,b)=\nu^3\circ a_{\lambda}(\chi_i)(x,y).$$ Thus, $\alpha_{\phi_*\eta_1,\chi_i}=\mu\circ a_{\lambda}(\chi_i)$, i.e for any $\xi\in \mathcal{M}^{\prime}$, $2$-cocycles $\alpha_{\phi_*\eta_1,\xi}$ and $\alpha_{\lambda,\xi}$ represent the same cohomology class in $\mathcal{H}_{CP}$. Therefore, the infinitesimal deformation $\lambda$ of the Courant pair $(A,L)$ is equivalent to the push-out $\phi_*\eta_1$. 

Next, we need to show that the map $\phi$ is uniquely defined. Let $\psi:C_1\rightarrow R$ be an arbitrary $\mathbb{K}$-linear map such that $\psi(1)=1$ and $\epsilon\circ \psi$ is the canonical augmentation in the $\mathbb{K}$-algebra $C_1$. Let $\psi_*\eta_1$ be equivalent to the infinitesimal deformation $\lambda$ of $L$ with base $(R,\mathcal{M})$, then $a_{\psi_*\eta_1}=a_{\lambda}$ and we get the following equations:
\begin{align*}
(1\otimes a)._{\psi_{*}\eta_1}(1\otimes b)
&=1\otimes ab+\sum_j x_j\otimes (\sum_i\chi_j(\psi(g_i))\nu^1(h_i))(a,b)\\
\mu_{\psi^*\eta_1}(1\otimes x)(1\otimes a)&=1\otimes \mu(x)(a)+\sum_j x_j\otimes (\sum_i\chi_j(\psi(g_i))\nu^2(h_i))(x,a)\\
 [1\otimes x,1\otimes y]_{\psi_{*}\eta_1}&=1\otimes[x,y]+\sum_j x_j\otimes (\sum_i\chi_j(\psi(g_i))\nu^3(h_i))(x,y).
 \end{align*}

Therefore, $\alpha_{\psi_*\eta_1}(\chi_j)=\nu(\sum_i\chi_j(\psi(g_i))h_i)$, which implies that the cohomology class represented by $\alpha_{\psi_*\eta_1}(\chi_j)$ is given by
$$a_{\psi_*\eta_1}(\chi_j)=\sum_i\chi_j(\psi(g_i))h_i.$$
Since, $a_{\psi_*\eta_1}=a_{\lambda}$ and $a_{\lambda}(\chi_j)=\sum_i\chi_j(a_{\lambda}^{\prime}(g_i))h_i$, we get $\psi(g_i)=a_{\lambda}^{\prime}(g_i)$ for $1\leq i\leq m$, which implies $\psi=\phi$.
\end{proof}

\begin{Def}
Let $(R,\mathcal{M})$ be a complete local algebra. A formal deformation of the Courant pair $(A,L)$ with base $R$ is a Courant pair structure on the pair $$\big(R\otimes A, R\otimes L\big)
:=\big(lim_{n\rightarrow \infty}((R/\mathcal{M}^n)\otimes A), lim_{n\rightarrow \infty}((R/\mathcal{M}^n)\otimes L)\big)$$
which is the projective limit of deformations of $(A,L)$ with base $(R/\mathcal{M}^n,\mathcal{M}^{n+1} )$, such that the following map:
$$\big(\epsilon\otimes id_A,\epsilon\otimes id_L\big):\big(R\otimes A, R\otimes L\big)\rightarrow (\mathbb{K}\otimes A,\mathbb{K}\otimes L)$$
is a homomorphism of Courant pairs over $R$. 
\end{Def}

\begin{Def}
A formal deformation $\eta$ of a Courant pair $(A,L)$ with base $C$ is called a versal deformation if 
\begin{itemize}
\item For any formal deformation $\lambda$ of $(A,L)$ with base $R$, then there exist an algebra homomorphism $\phi:C\rightarrow R$ such that $\lambda\cong \phi_*\eta$.
\item Let $\mathcal{M}$ be the maximal ideal in $R$ and $\mathcal{M}^2=0$, then $\phi$ is unique.
\end{itemize}
\end{Def}

\subsection{Extension of a deformation of Courant pair} 
Let $\lambda:=(~\bigcdot_{\lambda},\mu_{\lambda},[-,-]_{\lambda})$ be a deformation of Courant pair $(A,L)$ with a finite dimension local algebra $(R,\mathcal{M})$. 

From Theorem \ref{CorresHarr}, $H^2_{Harr}(R,\mathbb{K})$ corresponds bijectively to the set of isomorphism classes of extensions of $R$ by $\mathbb{K}$. Let $[\psi]\in H^2_{Harr}(R,\mathbb{K})$ corresponds to the equivalence class of 1-dimensional extension of $R$ represented by 
$$0\rightarrow \mathbb{K}\xrightarrow{i}S\xrightarrow{p}R.$$
Let us define the following maps:
\begin{itemize}
\item $(I_L,I_A):=(i\otimes id_L,i\otimes id_A):(\mathbb{K}\otimes L,\mathbb{K}\otimes A)\rightarrow (S\otimes L,S\otimes A)$,
\item $(P_L,P_A):=(p\otimes id_L,p\otimes id_A):(S\otimes L,S\otimes A)\rightarrow (R\otimes L,R\otimes A)$, and
\item $(E_L,E_A):=(\epsilon_S\otimes id_L,\epsilon_S\otimes id_A): (S\otimes L,S\otimes A)\rightarrow (\mathbb{K}\otimes L,\mathbb{K}\otimes A).$
\end{itemize}
Here, $\epsilon_S=\epsilon\circ p$ is an augmenation of $S$. Fix a section of the $\mathbb{K}$-linear map $p:S\rightarrow R$, then $S\cong R\oplus \mathbb{K}$ and the isomorphism is given by $s\mapsto (p(s),i^{-1}(s-q(p(s))))$. Denote the inverse image of $(r,k)\in R\oplus \mathbb{K}$ simply by $(r,k)_q$. Moreover, 
the algebra multiplication in $S$ is given by:
$$(r_1,k_1)_q(r_2,k_2)_q=(r_1r_2, \epsilon(r_1)k_2+\epsilon(r_2)k_1+\psi(r_1,r_2))$$ 

Let $\{m_i:1\leq i\leq n \}$ be a basis of the maximal ideal $\mathcal{M}_R$ in $R$ and $\{\xi_i:1\leq i\leq n \}$ be the dual basis of $\mathcal{M}_R^{\prime}$. Then by equations \eqref{asso}-\eqref{brac}, the deformation $\lambda=(\bigcdot_{\lambda},\mu_{\lambda},[-,-]_{\lambda})$ is given as follows:
\begin{align*}
(1\otimes a)\bigcdot_{\lambda}(1\otimes b)&=1\otimes ab+\sum_{i=1}^{n} m_i\otimes \alpha^1_{\lambda,\xi_i}(a,b),\\
\mu_{\lambda}(1\otimes x)(1\otimes a)&=1\otimes\mu(x)(a)+\sum_{i=1}^{n}m_i\otimes \alpha^2_{\lambda,\xi_i}(x,a),\\
[1\otimes x,1\otimes y]_{\lambda}&=1\otimes[x,y]+\sum_{i=1}^{n} m_i\otimes \alpha^3_{\lambda,\xi_i}(x,y).
\end{align*}
Denote $\psi^k_i:=\alpha^k_{\lambda,\xi_i}$. Note that $S$ is also a local algebra with maximal ideal $\mathcal{M}_S=p^{-1}(\mathcal{M}_R)$. Also $\{n_i:1\leq i\leq n+1 \}$ is a basis of $\mathcal{M}_S$, where $n_i=(m_i,0)_q$ for $1\leq i\leq n$ and $n_{n+1}=(0,1)_q$. 

For any fixed element $(\psi^1_{n+1},\psi^2_{n+1},\psi^3_{n+1})\in C^2_{tot}(A,L)$, let us define the following $S$-bilinear maps:
\begin{enumerate}
\item A product $\bigcdot_s:(S\otimes A)^{\otimes 2}\rightarrow S\otimes A$ given by
$$(s_1\otimes a)\bigcdot_s(s_2\otimes b)=s_1 s_2\otimes ab+\sum_i^{n+1} s_1s_2n_i\otimes \psi^1_i(a,b).$$
\item An action $\mu_s:(S\otimes L)\otimes (S\otimes A)\rightarrow S\otimes A$ given by 
$$\mu_s(s_1\otimes x)(s_2\otimes a)=s_1s_2\otimes\mu(x)(a)+\sum_{i=1}^{n+1}s_1s_2n_i\otimes \psi^2_i(x,a).$$
\item And a bracket $[-,-]_s:(S\otimes L)^{\otimes 2}\rightarrow S\otimes L$ given by $$[s_1\otimes x,s_2\otimes y]_s=s_1s_2\otimes[x,y]+\sum_{i=1}^{n+1} s_1s_2n_i\otimes \psi^3_i(x,y).
$$
\end{enumerate}

\begin{Rem}\label{S-Op-Prop}
The $S$-bilinear maps in the triplet $(~\bigcdot_s, \mu_s,[-,-]_s)$ satisfy the following consitions:
\begin{align*}
P_A(a^s_1\bigcdot_s~ a^s_2)&=P_A(a^s_1)\bigcdot_{\lambda}P_A(a^s_2),\\
I_A(a)\bigcdot_s~ a^s&=I_A(a\bigcdot_{\lambda}E_A(a^s)),\\
P_A(\mu_s(l^s)(a^s))&=\mu_{\lambda}(P_L(l^s))(P_A(a^s)),\\
\mu_s(I_L(l))(a^s)&=I_A(\mu_{\lambda}(l)(E_A(a^s))),\\ 
P_L[l^s_1,l^s_2]&=[P_L(l^s_1),P_L(l^s_2)]_{\lambda},\\
[I_L(l),l^s]_s&=I_L[l,E_L(l^s)]_{\lambda}.
\end{align*}
\end{Rem} 

Let us define a map $\phi_A:(S\otimes A)^{\otimes 3}\rightarrow S\otimes A$ by
$$\phi_A(a^s_1,a^s_2,a^s_3)=a^s_1\bigcdot_s(a^s_2\bigcdot_s a^s_3)-(a^s_1\bigcdot_s a^s_2)\bigcdot_s a^s_3.$$

Next, by using Remark \ref{S-Op-Prop}, let us observe that $\phi_A(a^s_1,a^s_2,a^s_3)\in Ker(P_A)$ for all $a^s_1,a^s_2,a^s_3\in S\otimes A$. Let $a^s_k\in Ker(E_A)=Ker(\epsilon_s)\otimes A=\mathcal{M}_s\otimes A$ for some $1\leq k\leq 3$, then we can write $$a^s_k=\sum_{i=1}^{n+1}(n_i\otimes a_i)$$ for some $a_i\in A$, and $1\leq i\leq n+1$. Now it is easy to see that $\phi_A(a^s_1,a^s_2,a^s_3)=0$. Therefore, $\Phi_A$ induces the following map:
$$\bar{\phi_A}:\big(S\otimes A/Ker(E_A)\big)^{\otimes 3}\rightarrow Ker(P_A).$$ 

Note that since $E_A$ is a surjective map, we have an isomorphism $$\alpha_A:A \rightarrow \frac{S\otimes A}{Ker(E_A)}$$
defined by $\alpha_A(a)=1\otimes a+ Ker(E_A)$. Moreover $Ker(P_A)=\mathbb{K}i(1)\otimes A\cong A$, where the isomorphism (denoted by $\beta_A$) is given by:
$\beta_A(k n_{n+1}\otimes a)=k.a$. Therefore, we get $\mathbb{K}$-linear map 
$$\theta_A: A^{\otimes 3}\rightarrow A$$
such that $\theta_A=\beta_A\circ \bar{\phi_A}\circ \alpha_A^{\otimes 3}$ and we have the following relation:
$$n_{n+1}\otimes \theta_A(a_1,a_2,a_3)=\phi_A(1\otimes a_1,1\otimes a_2,1\otimes a_3)$$  

Similarly, let us define the following linear maps: \\
\begin{enumerate}
\item $\phi_1:(S\otimes L)\otimes (S\otimes A)^{\otimes 2}\rightarrow S\otimes A$ by
$$\phi_1(l^s,a^s_1,a^s_2)=\mu_s(l^s)(a^s_1\bigcdot_s a^s_2)-\mu_s(l^s)(a^s_1)\bigcdot_s a^s_2-a^s_1\bigcdot_s\mu_s(l^s)( a^s_2).$$
\item $\phi_2:(S\otimes L)^{\otimes 2}\otimes (S\otimes A)\rightarrow S\otimes A$ by
$$\phi_1(l^s_1,l^s_2,a^s)=\mu_s([l^s_1,l^s_2])(a^s)-\mu_s(l^s_1)(\mu_s(l^s_2)(a^s)-\mu_s(l^s_2)(\mu_s(l^s_1)( a^s)).$$ 
\item $\phi_L:(S\otimes L)^{\otimes 3}\rightarrow S\otimes L$ by
$$\phi_L(l^s_1,l^s_2,l^s_3)=[l^s_1,[l^s_2, l^s_3]_s]_s-[[l^s_1,l^s_2]_s,l^s_3]_s-[[l^s_1,l^s_3]_s, l^s_2)]_s.$$ 
 \end{enumerate}
Similar to the above discussion for $\phi_A$ and induced map $\theta _A$, we get the following induced linear maps by $\phi_1,~\phi_2,$ and $\phi_L$, respectively:
\begin{enumerate}
\item A linear map $\theta_1:L\otimes A^2\rightarrow A$ satisfying $$n_{n+1}\otimes \theta_1(l,a_1,a_2)=\theta_1(1\otimes l,1\otimes a_1,1\otimes a_2)$$  
\item A linear map $\theta_2:L^2\otimes A\rightarrow A$ satisfying  $$n_{n+1}\otimes \theta_2(l_1,l_2,a)=\theta_2(1\otimes l_1,1\otimes l_2,1\otimes a),$$
\item A linear map $\theta_L: L^3\rightarrow L$ satisfying the relation:
$$n_{n+1}\otimes \theta_L(l_1,l_2,l_3)=\phi_L(1\otimes l_1,1\otimes l_2,1\otimes l_3)$$  
\end{enumerate} 

The $3$-cochain $(\theta_A,\theta_1,\theta_2,\theta_L)$ is a cocycle. In fact, if $\beta_A: Ker(P_A)\rightarrow A$ and $\beta_L:Ker(P_L)\rightarrow L$ are $\mathbb{K}$-linear isomorphisms, then we get the following:
\begin{itemize}
\item $\beta_A^{-1}\big(\delta_H\theta_A(a,b,c,d)\big)=0$
\item $\beta_A^{-1}\big(\delta_L\theta_A(x,a,b,c)+\delta_H\theta_1(x,a,b,c)\big)=0$
\item $\beta_A^{-1}\big(\delta_L\theta_1(x,y,a,b)+\delta_H\theta_2(x,y,a,b)\big)=0$
\item $\beta_A^{-1}\big(\delta_L\theta_2(x,y,z,a)+\delta_H\theta_L(x,y,z,a)\big)=0$
\item $\beta_L^{-1}\big(\delta_L\theta_L(x,y,z,t)\big)=0$
\end{itemize} 
for any $x,y,z,t\in L$, and $a,b,c,d\in A$, i.e. $\delta_{tot}(\theta_A,\theta_1,\theta_2,\theta_L)=0$. 

Next, we show that the cohomology class of $(\theta_A,\theta_1,\theta_2,\theta_L)$ is independent of the choice of $\psi_{n+1}\in C^2_{tot}(A,L)$. Let $(\bigcdot_s, \mu_s,[-,-]_s)$, and $(\bigcdot_s^{\prime}, \mu_s^{\prime},[-,-]_s^{\prime})$ be two Courant pair structures on $(S\otimes A, S\otimes L)$ extending the deformation $\lambda=(._{\lambda},\mu_{\lambda},[-,-]_{\lambda})$. Let $3$-cochains $\Theta:=(\theta_A,\theta_1,\theta_2,\theta_L)$ and $\Theta^{\prime}:=(\theta_A^{\prime},\theta_1^{\prime},\theta_2^{\prime},\theta_L^{\prime})$ be the corresponding $3$-cocycles.
Then define a $2$-cochain $\gamma$ as follows:
$$(\gamma_{1},\gamma_2,\gamma_3):=(\bigcdot_s-\bigcdot_s^{\prime},\mu_s- \mu_s^{\prime},[-,-]_s-[-,-]_s^{\prime})$$

Then similar to the above discussion it follows that $\delta_{tot}(\gamma)=\Theta-\Theta^{\prime}$. Thus we get a map $\Theta_{\lambda}:H^2_{Harr}(R,\mathbb{K})\rightarrow H^3_{CP}(A,L)$ defined by $\Theta_{\lambda}([\psi])=[\Theta]$ for any class $[\psi]\in H^2_{Harr}(R,\mathbb{K})$. The map $\Theta_{\lambda}$ is called the obstruction map. Finally, we have the following result formulating a necessary and sufficient condition for extending the deformation $\lambda$ of $(A,L)$ (with base $(R,\mathcal{M}_R)$).

\begin{Thm}\label{1dimext}
Let $\lambda$ be a deformation of the Courant pair $(A,L)$ with base $(R,\mathcal{M}_R)$, and $[\psi]\in H^2_{Harr}(R,\mathbb{K})$ corresponds to a $1-$dimensional extension $S$ of $\mathbb{K}$-algebra $R$. Then $\lambda$ extends to a deformation of $(A,L)$ with base $(S,\mathcal{M}_S)$ if and only if $\Theta_{\lambda}([\psi])=0$.
\end{Thm} 

\begin{proof}
Let $\Theta_{\lambda}([\psi])=0$ for $[\psi]\in H^2_{Harr}(R,\mathbb{K})$. Also, let $[\psi]$ corresponds to a $1$-dimensional extension of $R$ given by:
$$0\rightarrow \mathbb{K}\xrightarrow{i}S\xrightarrow{p} R.$$
Let $(~\bigcdot_s,\mu_s,[-,-]_s)$ be a triplet of $S$-bilinear maps on the pair $(S\otimes A,S\otimes L)$ extending $\lambda$ and $\Theta:=(\theta_A,\theta_1,\theta_2,\theta_L)$ be the associated $3$-cocycle. Then $\Theta([\psi])=[\Theta]=0$, i.e $\Theta= \delta_{tot}(\gamma)$ for some $\gamma\in C^2_{tot}(A,L)$. Let us define the following:
\begin{itemize}
\item A product $\bigcdot_s^{\prime}:(S\otimes A)\otimes_S (S\otimes A) \rightarrow S\otimes A$, which is defined by $$(a^s_1,a^s_2)\mapsto a^s_1\bigcdot_s a^s_2- I_A(\gamma_1(E_A(a^s_1),E_A(a^s_2)))$$
for any $a^s_1,~a^s_2\in S\otimes A$.
\item An action  $\mu_s^{\prime}:(S\otimes L)\otimes(S\otimes A)\rightarrow S\otimes A$, which is defined by  $$(l^s,a^s)\mapsto \mu_s(l^s)(a^s)- I_A(\gamma_2(E_L(l^s),E_A(a^s)))$$
for any $l^s\in S\otimes L$, and $a^s\in S\otimes A$.

\item A bracket $[-,-]_s^{\prime}:(S\otimes L)\otimes_S (S\otimes L)\rightarrow S\otimes L$ which is defined by
 $$(l^s_1,l^s_2)\mapsto [l^s_1,l^s_2]_s- I_L(\gamma_3(E_L(l^s_1),E_L(l^s_2))).$$
for any $l^s_1,l^s_2\in S\otimes L$.
\end{itemize}

Thus we get a triplet $(~\bigcdot_s^{\prime}, \mu_s^{\prime},[-,-]_s^{\prime})$. Let us denote the associated $3$-cochain to the triplet $(~\bigcdot_s^{\prime}, \mu_s^{\prime},[-,-]_s^{\prime})$, by $\Theta^{\prime}$. Then by definition of the triplet $(~\bigcdot_s^{\prime}, \mu_s^{\prime},[-,-]_s^{\prime})$, we get $$\Theta^{\prime}-\Theta=-\delta_{tot}(\gamma)=-\Theta,$$ 
which implies $\Theta^{\prime}=0$. Therefore, the triplet $(~\bigcdot_s^{\prime}, \mu_s^{\prime},[-,-]_s^{\prime})$ is a Courant pair structure on $(S\otimes A, S\otimes L)$, which extends the deformation $\lambda$. The converse part  follows vacuously. 
\end{proof}

\begin{Rem}
 Let $\lambda $ be a deformation of Courant pair $(A,L)$ with base $R$ and $\mathbb{K}$-algebra $S$ is a $1$-dimensional extension of $R$:
\begin{equation}\label{extension}
0\rightarrow \mathbb{K} \xrightarrow{i} S\xrightarrow{p} R\rightarrow 0.
\end{equation}
 If $u:S\rightarrow S$ is an automorphism of this extension which corresponds to a class $\psi\in H^1_{Harr}(R,\mathbb{K})$, and $\lambda^{\prime}$ be a deformation of $(A,L)$ such that $p_*\lambda^{\prime}=\lambda$, then the $2$-cochain
 $$(._{u_*\lambda^{\prime}},\mu_{u_*\lambda^{\prime}},[-,-]_{u_*\lambda^{\prime}})-(._{\lambda^{\prime}},\mu_{\lambda^{\prime}},[-,-]_{\lambda^{\prime}})$$   
 is a cocycle in the class $d\lambda(\psi)$. Thus, if the differential $d\lambda: H^1_{Harr}(R,\mathbb{K})\rightarrow \mathcal{H}_{CP}$ is a  surjective map and a deformation $\lambda^{\prime}$ of $(A,L)$ with base $S$, satisfying $p_*\lambda^{\prime}=\lambda$, exists, then it is unique up to an equivalence of deformations and an automorphism of the extension \eqref{extension}.  
\end{Rem}

Let $M$ be a $R$-module (which is finite dimensional as a $\mathbb{K}$-module) such that $r.m=0$ for any $r\in \mathcal{M}_R$ and $m\in M$. Then consider an extension of $R$ by $M$ given by 
$$0\rightarrow M\xrightarrow{i} S\xrightarrow{p} R\rightarrow 0.$$

The equivalence class of this extension corresponds to an element $[\psi]\in H^2_{Harr}(R,M)$. An analogous calculation to the $1$-dimensional case yields a cocycle $\Theta_M\in C^3_{tot}(A,L;M\otimes A, M\otimes L)$ with the cohomology class $[\Theta_M]\in H^3_{CP}(A,L;M\otimes A, M\otimes L)=M\otimes H^3_{CP}(A,L)$. Then the obstruction map is given by:
$$\Theta_{\lambda}: H^2_{Harr}(R,M)\rightarrow M\otimes H^3_{tot}(A,L)$$
assigning $[\psi]\rightarrow [\Theta_M]$, and we have the following generalisation of the Theorem \ref{1dimext}:

\begin{Thm}\label{findimext}
Let $\lambda$ be a deformation of a Courant pair $(A,L)$ with base $(R,\mathcal{M}_R)$ and $\mathbb{K}$-algebra $S$ is an extension of $R$ by a finite dimensional $R$-module $M$:
$$0\rightarrow M\xrightarrow{i} S\xrightarrow{p} R\rightarrow 0,$$
corresponding to a class $[\psi]\in H^2_{Harr}(R,M)$. Then a deformation $\mu$ of $(A,L)$ with base $(S,\mathcal{M}_S)$ satisfying $p_*\mu=\lambda$ exists if and only if $\Theta_{\lambda}([\psi])=0$.
\end{Thm} 

\subsection{Construction of Versal deformation of a Courant pair}

Let $(A,L)$ be a Courant pair such that $ H^2_{CP}(A,L)$ is finite dimension vector space. Define $C_0:=\mathbb{K}$, $C_1:=\mathbb{K}\oplus \mathcal{H}_{CP}^{\prime}$ and the multiplication in $C_1$ by: $(k_1,h_1)\cdot(k_2,h_2)=(k_1\cdot k_2, k_1h_2+k_2h_1)$. This multiplication makes $C_1$ a $\mathbb{K}$-algebra. Now, consider the following extension:
$$0\rightarrow \mathcal{H}_{CP}^{\prime}\xrightarrow{i} C_1\xrightarrow{p}C_0\rightarrow 0,$$
 where $p$ is the projection map. Let $\eta_1$ be the universal infinitesimal deformation with base $C_1$. Let us assume that for $k\geq 1$, we have a finite dimensional local algebra $C_k$ and a deformation $\eta_k$ of $(A,L)$ with base $C_k$. Let us consider the vector space of $2$-chains: $Ch_2(C_k)$ in Harrison complex of $C_k$ and define a linear map:
$$\nu:H^2_{Harr}(C_k,\mathbb{K})\rightarrow (Ch_2(C_k))^{\prime}$$
mapping a cohomology class to a cocycle representing the class. Let us denote the dual of the map $\nu$ by
$f_{k}:Ch_2(C_k)\rightarrow H^2_{Harr}(C_k,\mathbb{K})^{\prime}.$
Then $f_k$ is a cocycle representing a cohomology class in the $2$nd cohomology space of $C_k$ with coefficients in $H^2_{Harr}(C_k;\mathbb{K})^{\prime}$. By Proposition \ref{CorresHarr}, $f_k$ corresponds to an equivalence class of the following extension of $C_k$:
$$0\rightarrow H^2_{Harr}(C_k,\mathbb{K})\xrightarrow{j_{k+1}}\bar{C}_{k+1}\xrightarrow{q_{k+1}}C_k \rightarrow 0$$

Let us consider the extension of the deformation $\eta_k$ of the Courant pair $(A,L)$ with base $C_k$ to the above extension $\bar{C}_{k+1}$. The associated obstruction corresponding to this extension $\Theta([f_{k}])$ gives a linear map 
$$\varphi_k:H^2_{Harr}(C_k,\mathbb{K})\rightarrow H^3_{CP}(A,L).$$ 
Let the dual of the map $\varphi_k$ is given by 
$$\varphi_k^{\prime}:H^3_{CP}(A,L)^{\prime}\rightarrow H^2_{Harr}(C_k,\mathbb{K})^{\prime}.$$ 
Thus we have an extension of $C_k$ by the cokernel of the map $\varphi_k^{\prime}$: 
$$0\rightarrow coker(\varphi_k^{\prime})\xrightarrow{i_{k+1}}\bar{C}_{k+1}/j_{k+1}\circ \varphi_k^{\prime}(H^3_{CP}(A,L)^{\prime})\xrightarrow{p_{k+1}}C_k \rightarrow 0$$
Note that $coker(\varphi_k^{\prime})=(ker(\varphi_k))^{\prime}$, also let us denote $C_{k+1}:=\bar{C}_{k+1}/j_{k+1}\circ \varphi_k^{\prime}(H^3_{CP}(A,L)^{\prime})$, then it yields the following extension:
\begin{equation}\label{ext}
(ker(\varphi_k))^{\prime}\xrightarrow{i_{k+1}}C_{k+1}\xrightarrow{p_{k+1}}C_k \rightarrow 0
\end{equation} 
The obstruction map corresponding to the extension \eqref{ext} is the restriction of the map $\varphi_k$ on $Ker(\varphi_k)$ and then by using Theorem \ref{findimext}, we have the following proposition:

\begin{Prop}
The deformation $\eta_k$ of the courant pair with base $C_k$ extends to a deformation $\eta_{k+1}$ with base $C_{k+1}$, which is unique upto an isomorphism and automorphism of the extension \eqref{ext}.
\end{Prop}

Thus, using induction on $k$, we obtain a sequence of finite dimensional local algebras:
$$\mathbb{K}\xleftarrow{p_1}C_1\xleftarrow{p_2}\cdots\xleftarrow{p_k}C_k\xleftarrow{p_{k+1}}C_{k+1}\xleftarrow{p_{k+2}}\cdots$$
and deformations $\eta_k$ of the Courant pair $(A,L)$ with base $C_k$ such that $\eta_k=p_{k+1_*}\eta_{k+1}$. Denote the projective limit by $C:=\varprojlim_{k\rightarrow\infty}{C_k}$. Finally, the projective limit $\eta:=\varprojlim_{k\rightarrow\infty}{\eta_k}$ is a formal deformation of the Courant pair $(A,L)$ with the base $C:=\varprojlim_{k\rightarrow\infty}{C_k}$. 

\begin{Thm}
Let $(A,L)$ be a Courant pair and $dim\big( H^2_{CP}(A,L)\big)<\infty$, then the formal deformation $\eta$ with base $C$ is a versal deformation of the Courant pair $(A,L)$, where $\eta:=\varprojlim_{k\rightarrow\infty}{\eta_k}$, and the base $C:=\varprojlim_{k\rightarrow\infty}{C_k}$.
\end{Thm}

\begin{proof}
The proof follows by an analogous argument for the proof of Leibniz algebra case done in \cite{FMM}.
\end{proof}
\section{Deformation of  Poisson algebras as Courant pairs}

In this section, we deduce a bicomplex of a Poisson algebra by considering it as a Courant pair. Let $A$ be a Poisson algebra and $M$ be a Poisson $A$-module, then define a bicomplex as follows: for $p\geq 2$, and $q\geq 0$, set $C^{p,q}:=C^{p,q}(A;M)=Hom(A^q\otimes A^p,M)$ and for $p=1$, and $q\geq 0$, set  $C^{1,q}(A;M)=Hom(A^{q+1},M)$. The vertical and horizontal maps are given as follows:  

\begin{itemize}
\item For $p\geq 1$, the vertical map $\delta_H:C^{p,q}\rightarrow C^{p+1,q} $  is the Hochschild coboundary given by equation \eqref{Hochschild Coboundary}. 
\item The horizontal map $\delta_L$ is the Leibniz coboundary given by equation \eqref{Leibniz coboundary}.
\end{itemize}

Then the bicomplex of the Poisson algebra $A$ with coefficients in Poisson module $M$ is given by the following diagram:
 
\[
\begin{CD}
.... @. ....@. .... @. ....@. \\
@AAA @A\delta_H AA @A\delta_H AA  @A\delta_H AA \\
0 @>>> Hom(A^2,M) @>\delta_L>> Hom(A\otimes A^2,M) @>\delta_L>> Hom(A^2\otimes A^2,M)
@>\delta_L >> ....\\
@AAA @A{\delta_H} AA @A{\delta_H} AA @A{\delta_H} AA\\
M @>\delta_L>> Hom(A,M) @>\delta_L>> Hom(A^2,M)@>\delta_L>>Hom(A^3,M)@>\delta_L>>...\\
\end{CD}
\]

Next, note that $\delta_H\circ \delta_L=\delta_L\circ \delta_H$. Thus, we have a total cochain complex given as follows: for $n\geq 0$,
 $$C^n_{tot}(A;M)=\bigoplus _{p+q=n}C^{p,q}(A;M)~\mbox{and}~\delta_{tot}:C^n_{tot}( A;M )\rightarrow C^{n+1}_{tot}( A;M ),$$ whose restriction to $C^{p,q}( A;M )$ is $\delta_H+(-1)^p\delta_L$, for $p\geq 2,~q\geq 0$, and to $C^{1,q}( A;M )$ is $\delta_H+\delta_L$. Let us denote the associated cohomology to this total complex by $H^*_{CP}(A;M)$. If $M=A$, then denote the cohomology of the Poisson algebra with coefficients in itself by $H^*_{CP}(A)$.
 
 Note that category of Poisson algebra is a subcategory of the category of Courant pairs. In particular, if $A$ is a Poisson algebra, then $(A,A)$ is a Courant pair, where the Leibniz algebra homomorphism $\mu:A\rightarrow Der(A)$ is given by the adjoint action of $A$ on itself. Then the cohomology $H^*_{CP}(A)$ is a deformation cohomology for a Poisson algebra $A$, and we get analogous results to the Courant pair case in Section $3$. 

\subsection{Deformations of Poisson structures on Heisenberg Lie algebra}

 Let $H$ be the three-dimensional complex Lie algebra with a basis $\{e_1,e_2,e_3\}$ such that the Lie bracket is given as follows
 $$[e_1,e_2]=-[e_2,e_1]=e_3, ~[e_2,e_3]=[e_1,e_3]=0.$$
 
 Then $H$ is a three-dimensional complex Heisenberg Lie algebra. Now, recall the Goze-Remm classification \cite{Goze} of Poisson algebra structures on the Heisenberg Lie algebra. The isomorphism classes of Poisson algebra structures on $H$ are divided into two families:
 
 \begin{itemize}
 \item First there is a one-parameter family of Poisson algebra structures on $H$,
$$\mathcal{P}_1(\xi) := \{e_1.e_2 = \xi e_3\},$$
where $\xi\in \mathbb{C}$. Here, the underlying commutative associative product of Poisson algebra $\mathcal{P}_1(\xi)$ is given by: $e_1.e_2 = \xi e_3 =e_2.e_1,$ and all other products are $0$. One can linearly extend the product on arbitrary elements. 
\item The only other isomorphism class of Poisson algebra structure on $H$ is the Poisson algebra 
$$\mathcal{P}_2 := \{e_1^2 = e_3\}.$$

\end{itemize}  

Now, let us fix an ordered basis of $H$ given by: $\mathfrak{B}_1:=\{e_1,e_2,e_3\}$, and an ordered basis of $H\otimes H$ given by:$$\mathfrak{B}_2:=\{e_1\otimes e_1,e_1\otimes e_2,e_1\otimes e_3,e_2\otimes e_1,e_2\otimes e_2,e_2\otimes e_3,e_3\otimes e_1,e_3\otimes e_2,e_3\otimes e_3\}.$$

\begin{Exm}\label{P1}
Let us consider the one-parameter family of Poisson algebra structures on $H$ given by 
$$\mathcal{P}_1(\xi) := \{e_1.e_2 = \xi e_3\},$$

 Let us fix $\xi\in \mathbb{C}$ and consider the Leibniz $2$-cocycles of the Lie algebra $H$ with coefficients in its adjoint representation. Let $\phi$ be a Leibniz $2$-cocycle, i.e. $\delta_L(\phi)=0$. Then the matrix of $\phi$ in terms of the ordered basis $\mathfrak{B}_2$ of $H\otimes H$ and the ordered basis $\mathfrak{B}_1$ of $H$, is of the following form:
\[ \left( \begin{array}{ccccccccc}
0 & x_9 & x_1 & -x_9 & 0 & -x_7 & -x_1 & x_7 & 0\\
0 & x_4 & x_5 & -x_4 & 0 & -x_1 & -x_5 & x_1 & 0\\
x_2 & x_{10} & x_6 & x_{11} & x_3& -x_8 & -x_6  & x_8 & 0\end{array} \right)\]

 Next, let us consider the underlying associative algebra and let $\psi$ be a Hochschild $2$-cocycle with coefficients in itself. In other words,
$\psi:H\otimes H\rightarrow H$ is a linear map and
$\delta_H (\psi)=0$, where
$$\delta_H \psi(e_i,e_j,e_k)=e_i.\psi(e_j,e_k)-\psi(e_i.e_j,e_k)+\psi(e_i,e_j.e_k)-\psi(e_i,e_j).e_k.$$
In terms of the ordered basis $\mathfrak{B}_2$ of the vector space $H\otimes H$, and ordered basis $\mathfrak{B}_1$ of $H$, the matrix of $\psi$ is of the form:
\[ \left( \begin{array}{ccccccccc}
y_1 & y_4 & 0 & y_4 & y_9 & 0 & 0 & 0 & 0\\
y_5 & y_1-y_2 & 0 & y_1-y_2 & y_4+y_3 & 0 & 0 & 0 & 0\\
y_6 & y_7 & y_2 & y_8 & y_{10} & y_3 & y_2 & y_3 & 0\end{array} \right)\] 

If $(\psi,\phi)$ is a $2$-cocycle in the total complex of the Poisson algebra $\mathcal{P}_1(\xi)$. Then $\phi$ and $\psi$ satisfy the following identities:
$$\delta_L(\phi)=0,~~\delta_H(\psi)=0,~~\mbox{and}~~{\delta_H}(\phi)+ \delta_L(\psi)=0,$$
which in turn gives the following form of the Leibniz $2$-cocycle $\phi$ and the Hochschild $2$-cocycle $\psi$:
\[ \phi=
\left( \begin{array}{ccccccccc}
0 & Y_{43} & 0 & -Y_{43} & 0 & 0 & 0 & 0 & 0\\
0 & -Y_{12} & 0 & Y_{12} & 0 & 0 & 0 &0 & 0\\
X_1 & X_2 & Y_{12} & X_3 & X_4& Y_{43} & -Y_{12}  & -Y_{43} & 0\end{array} \right)\]
 and 
\[ \psi=
\left( \begin{array}{ccccccccc}
y_1 & y_4 & 0 & y_4 & 0 & 0 & 0 & 0 & 0\\
0 & y_1-y_2 & 0 & y_1-y_2 & y_4+y_3 & 0 & 0 & 0 & 0\\
y_5 & y_6 & y_2 & y_7 & y_8 & y_3 & y_2 & y_3 & 0\end{array} \right)\] 
Here, $Y_{12}=\frac{2y_2-y_1}{2\xi}$, $Y_{43}=\frac{y_4-y_3}{2\xi}$, and $X_i,y_j\in \mathbb{C}$ for $~1\leq i\leq 4, ~1\leq j\leq 8$. Let us consider the following $2$-cocycles in the total complex:

\begin{itemize}
\item $\alpha_i:=(0,\Phi_i)$ with $X_i=1$, $y_k=0$ and $X_j=0$ for $i\neq j,~1\leq k\leq 8$,
\item $\alpha_i:=(\Psi_i,0)$ with $y_i=1$, $y_k=0$ and $X_j=0$ for $i\in \{5,6,7,8\},~i\neq k,~1\leq j\leq 4$, and 
\item $\beta_t:=(\Psi_t,\bar{\Phi}_t)$ with $y_t=1,~y_k=0$ and $X_j=0$ for $t\in \{1,2,3,4\},~t\neq k,~1\leq j\leq 4$
\end{itemize}

Then the set $\{\alpha_i,\beta_t:1\leq i\leq 8, 1\leq t\leq 4\}$ is a basis of $Z^2_{tot}(\mathcal{P}_1(\xi))$. If $(\psi,\phi)$ is a $2$-coboundary in the total complex, i.e. there exists $\gamma\in Hom(H,H)$ such that $\delta_L(\gamma)=\phi$, and $\delta_H(\gamma)=\psi$. Then matrix forms of $\phi$, and $\psi$ are given as follows:

\[ \phi=
\left( \begin{array}{ccccccccc}
0 & \frac{y_4}{\xi} & 0 & -\frac{y_4}{\xi} & 0 & 0 & 0 & 0 & 0\\
0 & -\frac{y_2}{\xi} & 0 & \frac{y_2}{\xi} & 0 & 0 & 0 &0 & 0\\
0 & y_6& \frac{y_2}{\xi} & -y_6 & 0& \frac{y_4}{\xi} & -\frac{y_2}{\xi}  & -\frac{y_4}{\xi} & 0\end{array} \right)\]
 and 
\[ \psi=
\left( \begin{array}{ccccccccc}
0 & y_4 & 0 & y_4 & 0 & 0 & 0 & 0 & 0\\
0 & -y_2 & 0 & -y_2 & 0 & 0 & 0 & 0 & 0\\
y_5 & y_6 & y_2 & y_6 & y_8 & -y_4 & y_2 & -y_4 & 0\end{array} \right)\] 
Let $\alpha_i^{\prime}$ be the total $2$-coboundary obtained by placing $y_i=1$, and $y_j=0$ for $j\neq i$, then we have the following identities:
\begin{enumerate}
\item $\alpha_5^{\prime}=\alpha_5,$
\item $\alpha_8^{\prime}=\alpha_8,$
\item $\alpha_2^{\prime}=\beta_2,$
\item $\alpha_4^{\prime}=\beta_4-\beta_3,$ and
\item $\alpha_6^{\prime}=\alpha_6+\alpha_7+\alpha_2-\alpha_3.$
\end{enumerate}

Thus a basis of the $2$nd cohomology space $\mathcal{H}_{CP}:=H^2_{CP}(\mathcal{P}_1(\xi))$ of the Poisson algebra $\mathcal{P}_1({\xi})$ is given by: $$\{[\alpha_1],[\alpha_2],[\alpha_3],[\alpha_6],[\alpha_7],[\beta_1],[\beta_3]\}$$
and hence $dim\mathcal{H}_{CP}=7$. Let $\{g_i:1\leq i\leq 7\}$ be the corresponding dual basis. Then the universal infinitesimal deformation $\eta_1:=(~\bigcdot_{\eta_1},[-,-]_{\eta_1})$ of the Poisson algebra $\mathcal{P}_1(\xi)$ with the base $(\mathbb{C}\oplus \mathcal{H}_{CP}^{\prime},\mathcal{H}_{CP}^{\prime})$ is given by:
\begin{align*}
[e_i,e_j]_{\eta_1}=& 1\otimes [e_i,e_j]+g_1\otimes \Phi_1(e_i,e_j)+g_2\otimes \Phi_2(e_i,e_j)+g_3\otimes \Phi_3(e_i,e_j)+\\
&g_6\otimes \bar{\Phi}_1(e_i,e_j)+g_7\otimes \bar{\Phi}_3(e_i,e_j),\\
e_i\bigcdot_{\eta_1}e_j=& 1\otimes e_i\bigcdot e_j+g_4\otimes \Psi_6(e_i,e_j)+g_5\otimes \Psi_7(e_i,e_j)+g_6\otimes \Psi_1(e_i,e_j)+\\
&g_7\otimes \Psi_3(e_i,e_j).\\
\end{align*}

Note that $2$nd Leibniz pair cohomology space $\mathcal{H}_{LP}:=H^2_{LP}(\mathcal{P}_1({\xi}))$ is spanned by the basis $\{[\alpha_6],[\alpha_7],[\beta_1],[\beta_3]\}$. Let $\{h_i:1\leq i\leq 3\}$ be the corresponding dual basis of $\mathcal{H}_{LP}^{\prime}$, then the universal infinitesimal deformation $\eta^{\prime}_1:=(~\bigcdot_{\eta^{\prime}_1},[-,-]_{\eta_1^{\prime}})$ of the Poisson algebra $\mathcal{P}_1(\xi)$ with base $(\mathbb{C}\oplus \mathcal{H}_{LP}^{\prime},\mathcal{H}_{LP}^{\prime})$ is given by:
\begin{align*}
[e_i,e_j]_{\eta^{\prime}_1}=& 1\otimes [e_i,e_j]+h_2\otimes \bar{\Phi}_1(e_i,e_j)+h_3\otimes \bar{\Phi}_3(e_i,e_j),\\
e_i\bigcdot_{\eta^{\prime}_1}e_j=& 1\otimes e_i\bigcdot e_j+h_1\otimes \Psi_6(e_i,e_j)+h_2\otimes \Psi_1(e_i,e_j)+h_3\otimes \Psi_3(e_i,e_j).\\
\end{align*}
\end{Exm}

\begin{Exm}\label{P2}
Let us consider the Poisson algebra structure on $H$ given by: 
$$\mathcal{P}_2 := \{e_1^2 = e_3\}.$$

Similar to the previous example  the matrix of a Leibniz $2$-cocycle $\phi$ in terms of the ordered basis $\mathfrak{B}_2$ of $H\otimes H$, and the ordered basis $\mathfrak{B}_1$ of $H$, is of the form:

\[ \left( \begin{array}{ccccccccc}
0 & x_9 & x_1 & -x_9 & 0 & -x_7 & -x_1 & x_7 & 0\\
0 & x_4 & x_5 & -x_4 & 0 & -x_1 & -x_5 & x_1 & 0\\
x_2 & x_{10} & x_6 & x_{11} & x_3& -x_8 & -x_6  & x_8 & 0\end{array} \right)\]
 
 If $\psi$ is a Hochschild $2$-cocycle of $\mathcal{P}_2$ with coefficients in itself, then in terms of $\mathfrak{B}_2$ and $\mathfrak{B}_1$, $\psi$ has the matrix form:

\[ \left( \begin{array}{ccccccccc}
y_1 & y_4 & 0 & y_4 & y_9 & 0 & 0 & 0 & 0\\
y_5 & y_1-y_2 & 0 & y_1-y_2 & y_4+y_3 & 0 & 0 & 0 & 0\\
y_6 & y_7 & y_2 & y_8 & y_{10} & y_3 & y_2 & y_3 & 0\end{array} \right)\] 

Let $(\psi,\phi)$ is a total $2$-cocycle in the total complex of the Poisson algebra $\mathcal{P}_2$. Then the following forms of matrices of $\phi$, and $\psi$ can be deduced:

\[ \phi=
\left( \begin{array}{ccccccccc}
0 & Y_{56} & 0 & -Y_{56} & 0 & 0 & 0 & 0 & 0\\
0 & X_2 & 0 & -X_2 & 0 & 0 & 0 &0 & 0\\
X_1 & X_4 & -y_2 & X_5 & X_3& -Y_{16} & y_2  & Y_{16} & 0\end{array} \right)\]
 and 
\[ \psi=
\left( \begin{array}{ccccccccc}
y_1 & y_4 & 0 & y_4 & 0 & 0 & 0 & 0 & 0\\
y_2 & y_6 & 0 & y_6 & 2y_4 & 0 & 0 & 0 & 0\\
y_3 & y_7 & y_5 & y_8 & y_9 & y_4 & y_5 & y_4 & 0\end{array} \right)\] 
where, $Y_{16}=2y_6-y_1$, $Y_{56}=y_6-y_5$. Thus, we get a basis of $Z^2_{tot}(\mathcal{P}_2)$ given by:\\ $\{\beta_3,\beta_4,\beta_7,\beta_8,\beta_9,\bar{\beta_1},\bar{\beta_2},\bar{\beta_5},\bar{\beta_6},\alpha_i:~1\leq i\leq 5\}$, where

\begin{itemize}
\item $\alpha_i=(0,\Phi_i)$ with $X_i=1$, $y_k=0$ and $X_j=0$ for $i\neq j,~1\leq i\leq 5,~1\leq k\leq 9$,
\item $\beta_i:=(\Psi_i,0)$ with $y_i=1$, $y_k=0$ and $X_j=0$ for $i\in \{3,4,7,8,9\},~k\neq i,~1\leq j\leq 5$, 
\item $\bar{\beta}_t:=(\Psi_t,\bar{\Phi}_t)$ with $y_t=1,~y_k=0$ and $X_j=0$ for $t\in \{1,2,5,6\},~t\neq k,~1\leq j\leq 5$.
\end{itemize}

Moreover, if $(\psi,\phi)$ is a $2$-coboundary in the total complex, then it follows that matrices of the Leibniz $2$-coboundary and the Hochschild $2$-coboundary are of the form:

\[ \phi=
\left( \begin{array}{ccccccccc}
0 & y_1 & 0 & -y_1 & 0 & 0 & 0 & 0 & 0\\
0 & y_2 & 0 & -y_2 & 0 & 0 & 0 &0 & 0\\
0 & X_4& -y_2 & -X_4 & 0& y_1 & y_2  & -y_1 & 0\end{array} \right)\]
 and 
\[ \psi=
\left( \begin{array}{ccccccccc}
y_1 & 0 & 0 & 0 & 0 & 0 & 0 & 0 & 0\\
y_2 & 0 & 0 & 0 & 0 & 0 & 0 & 0 & 0\\
y_3 & y_7 & -y_1 & y_7 & 0 & 0 & -y_1 & 0 & 0\end{array} \right)\] 

Next let us denote,
\begin{itemize}
\item $\alpha_4^{\prime}$:= the total $2$-coboundary obtained by placing $X_4=1,y_i=0$ for $i=1,2,3,7$;
\item $\beta_i^{\prime}:=$ total $2$-coboundary obtained by placing $y_i=1,X_4=0, y_j=0$ for $i=7,3,j\neq i;$ 
\item $\bar{\beta}_i^{\prime}:=$ total $2$-coboundary obtained by placing $y_i=1,X_4=0,y_j=0$ for $i=1,2,j\neq i$.
\end{itemize}   

Then we get the following identities:  
\begin{enumerate}
\item $\alpha_4^{\prime}=\alpha_4-\alpha_5,$
\item $\beta_7^{\prime}=\beta_7+\beta_8,$
\item $\beta_3^{\prime}=\beta_3,$
\item $\bar{\beta}_1^{\prime}=\bar{\beta}_1-\bar{\beta}_5,$ and
\item $\bar{\beta}_2^{\prime}=\bar{\beta}_2-\alpha_2.$
\end{enumerate}

 Therefore, the set $\{[\alpha_i],[\beta_4],[\beta_7],[\beta_9],[\bar{\beta_1}],[\bar{\beta_6}],~1\leq i\leq 4\}$ is a basis of the $2$nd cohomology space $\mathcal{H}_{CP}:=H_{CP}^2(\mathcal{P}_2)$. Let us denote the corresponding dual basis by: $\{g_i:1\leq i\leq 9\}$, then a universal infinitesimal deformation $\eta_1:=(\bigcdot_{\eta_1},[-,-]_{\eta_1})$ of the Poisson algebra $\mathcal{P}_2$ with the base $\mathbb{C}\oplus \mathcal{H}_{CP}^{\prime}$ can be expressed in terms of the basis of $2$nd cohomology space and the dual basis as follows:
\begin{align*}
[e_i,e_j]_{\eta_1}=& 1\otimes [e_i,e_j]+g_1\otimes \Phi_1(e_i,e_j)+g_2\otimes \Phi_2(e_i,e_j)+g_3\otimes \Phi_3(e_i,e_j)+\\
&g_4\otimes \Phi_4(e_i,e_j)+g_8\otimes \bar{\Phi}_1(e_i,e_j)+g_9\otimes \bar{\Phi}_6(e_i,e_j),\\
e_i\bigcdot_{\eta_1}e_j=& 1\otimes e_i\bigcdot e_j+g_5\otimes \Psi_4(e_i,e_j)+g_6\otimes \Psi_7(e_i,e_j)+\\
&g_7\otimes \Psi_9(e_i,e_j)+g_8\otimes \Psi_1(e_i,e_j)+g_9\otimes \Psi_6(e_i,e_j).\\
\end{align*}

Note that it easily follows that a basis of $2$nd Leibniz pair cohomology space $\mathcal{H}_{LP}:=H^2_{LP}(\mathcal{P}_2)$ of the Poisson algebra $\mathcal{P}_2$ is given by: $\{[\alpha_2],[\beta_4],[\beta_7],[\beta_9],[\bar{\beta_1}],[\bar{\beta_6}]\}.$ Let $\{h_i:1\leq i\leq 6\}$ be the corresponding dual basis, then the universal infinitesimal (Leibniz pair) deformation $\eta^{\prime}_1:=(~\bigcdot_{\eta^{\prime}_1},[-,-]^{\prime}_{\eta_1})$ of the Poisson algebra $\mathcal{P}_1(\xi)$ with base $(\mathbb{C}\oplus \mathcal{H}_{LP}^{\prime},\mathcal{H}_{LP}^{\prime})$ is given by:
\begin{align*}
[e_i,e_j]_{\eta^{\prime}_1}=& 1\otimes [e_i,e_j]+h_1\otimes \Phi_2(e_i,e_j)+h_6\otimes \bar{\Phi}_6(e_i,e_j),\\
e_i\bigcdot_{\eta^{\prime}_1}e_j=& 1\otimes e_i\bigcdot e_j+h_2\otimes \Psi_4(e_i,e_j)+h_3\otimes \Psi_7(e_i,e_j)+\\&h_4\otimes \Psi_9(e_i,e_j)+h_5\otimes \Psi_1(e_i,e_j)+h_6\otimes \Psi_6(e_i,e_j).\\
\end{align*}
\end{Exm}
\begin{Rem}
The above examples \ref{P1} and \ref{P2} show that the deformation of a Poisson algebra when viewed as a Courant pair differs from it's deformation as a Leibniz pair, even at the infinitesimal level. Moreover, by deforming a Poisson algebra as a Courant pair not only we get all the Leibniz pair deformations of the Poisson algebra but we can also get new deformations of this Poisson algebra as a Courant pair which is not a Leibniz pair. 
\end{Rem}

{\bf Conclusions:}
In this note, we discussed deformations of Courant pairs with a commutative algebra base. By the process of skew-symmetrization it yields deformation of  Leibniz pairs with commutative algebra base. We present an inductive construction of a versal deformation by starting from an universal infinitesimal deformation with the base of the deformation satisfying a given cohomological condition. If we consider a Poisson algebra example, then it is an object in both the category of Leibniz pairs and the category of Courant pairs. The deformation shows a difference even in the infinitesimal level and provides many more extra deformations from the Leibniz pair case. This computation shows that one may consider other examples of Leibniz pairs, which yield new deformations as Courant pair.  In \cite{FM}, such comparison of deformations is shown for Lie algebras considered in the larger category of Leibniz algebras. 

\newpage
\appendix

\section{Cohomology of associative algebras}
\begin{Def}
 Let $A$ be an associative algebra over $\mathbb{K}$, a bimodule $M$ over $A$ or, an $A$-bimodule is a  $\mathbb{K}$-module equipped with two actions (left and right) of $A$
$$ A\times M \longrightarrow M ~~\mbox{and}~ M\times A \longrightarrow M$$ such that 
$(a m) a^\prime=a (m a^\prime)$ for $a,a^\prime \in A$ and $m \in M$.

The actions of $A$ and $\mathbb{K}$ on $M$ are compatible, that is $(\lambda a) m=\lambda (a m)=a (\lambda  m)$ for $\lambda \in \mathbb{K},~a\in A ~and ~m \in M$. When $A$ has the identity $1$ we always assume that $1 m=m 1=m ~~\mbox{for}~ m \in M$.
\end{Def}
Given a bimodule $M$ over $A$, the  Hochschild cochain complex of $A$ with coefficients in $M$ is defined as follows.
Set $C^n(A,M)= Hom_\Kmath (A^{\otimes n},M)~; n\geq 0$ where $A^{\otimes n}=A \otimes \cdots \otimes A~(n\mbox{-copies})$.
Let $\delta_H:C^n(A,M) \rightarrow C^{n+1}(A,M)$ be the $\mathbb{K}$- linear map given by 
\begin{equation}\label{Hochschild Coboundary}
 \begin{split}
\delta_H f(a_1,a_2,\ldots,a_{n+1})&=a_1.f(a_2,\ldots,a_n)+\sum_{i=1}^{n}(-1)^i f(a_1,\ldots,a_i a_{i+1},\ldots,a_{n+1})\\
&+(-1)^{n+1} f(a_1,\ldots,a_{n-1},a_{n}).a_{n+1} 
\end{split}
\end{equation} 
Then $\delta_H^2=0$ and the complex $(C^{*}(A,M),\delta_H)$ is called the Hochschild complex of $A$ with coefficients in the $A$-bimodule $M$.
 When $M=A$, where the actions are given by algebra operation in $A$ we denote the complex $C^{*}(A,A)$ by $C^{*}(A)$. The graded space $C^*(A)$ is a differential graded Lie algebra (or DGLA in short ) (see \cite{G1}).
\begin{Def}
A graded Lie algebra $\mathfrak{L}$ is a graded module $\mathfrak{L}=\{L_i\}_{i\in \mathbb{Z}}$ together with a linear map of degree zero, $[-,-]:\mathfrak{L}\otimes \mathfrak{L}\longrightarrow \mathfrak{L}$, $x\otimes y\mapsto [x,y]$ satisfying
\begin{equation*}
 \begin{split}
  &(i)[x,y]=-(-1)^{|x||y|}[y,x] \mbox {(graded skewsymmetry)}\\
&(ii)(-1)^{|x||z|}[x,[y,z]]+(-1)^{|y||x|}[y,[z,x]]+(-1)^{|z||y|}[z,[x,y]]=0 ~~\mbox{( graded Jacobi identity )}\\
 \end{split}
\end{equation*}
for $x,y,z \in \mathfrak{L}$, where $|x|$ denotes the degree of $x$.

A differential graded Lie algebra is a graded Lie algebra equipped with a differential $d$ satisfying 
$$d[x,y]=[dx,y]+(-1)^{|x|}[x,dy].$$
\end{Def}
\begin{Rem}\label{DGLA}
The shifted Hochschild complex $C^{*}(A,A)=C^{*+1}(A, A)$ is a DGLA where the graded Lie bracket is also called the Gerstenhaber bracket. 

Let $\phi$ be a Hochschild $p$-cochain and $\theta$ be a Hochschild $q$-cochain. The Gerstenhaber bracket of $\phi$ and $\theta$ is the $(p+q-1)$-cochain defined by
$$[\phi,\theta]=\phi\circ\theta-(-1)^{(p-1)(q-1)}\theta\circ \phi,$$
where
$$\phi\circ\theta(a_1,\cdots,a_{p+q-1})=\sum_{i=0}^{p-1}(-1)^{i(q+1)}\phi(a_1,\cdots,a_i,\theta(a_{i+1},\cdots,a_{i+q}),a_{i+q+1},\cdots,a_{p+q-1}).$$

Also, note that if $\alpha_0$ is the associative multiplication, then $\delta_H$ can be written in terms of the Gerstenhaber bracket and the multiplication $\alpha_0$:
$$\delta_H \phi=(-1)^{|\phi|}[\alpha_0,\phi].$$
\end{Rem}

{\bf Harrison complex for commutative algebras}

Let $R$ be a commuative algebra over $\mathbb{K}$. Let $(\oplus_{q\geq 0} C_q(R),\delta)$ denotes the standard Hochschild complex, where $C_q(R):=R^{\otimes(q+1)}$ be an $R$-module. Let $Sh_q(R)$ be the $R$-submodule generated by the elements in $C_q(R)$ of the form:
$$s_p(r_1,\cdots,r_{q})=\sum_{(i_1,\cdots,i_q)\in Sh(p,q-p)} sgn(i_1,\cdots,i_q)(r_{i_1},\cdots,r_{i_q})$$
for any $r_i\in R, ~1\leq i\leq q$, and $0<p<q$. In fact, $(Sh_*(R),\delta)$ is a subcomplex of $(C_*(R),\delta)$. Hence, we have a chain complex which is called Harrison complex given by
$$Ch(R)=(\oplus_{q\geq 0}Ch_q(R), \delta)$$
where $Ch_q(R):=C_q(R)/Sh_q(R)$. 

Let $M$ be an $R$-module then the Harrison cochain complex with coefficients in $M$ is given by $(\oplus_{q\geq 0}Ch^q(R,M)=Hom(Ch_q(R),M),\partial)$. The cohomology of this cochain complex is called Harrison cohomology, i.e.
$$H^q_{Harr}(R,M)=H^q(Hom(Ch(R),M)).$$
 Recall that if $(R,\mathcal{M})$ is a commutative local algebra and $M$ is an $R$-module such that $\mathcal{M}.M=0$. Then we have following isomorphism:
 $$H^q_{Harr}(R,M)\cong H^q_{Harr}(R,\mathbb{K})\otimes M.$$
 
 \begin{Def}
 A $\mathbb{K}$-algebra $S$ is called an extension of $\mathbb{K}$-algebra $R$ by an $R$-module $M$ if we have an exact sequence of $\mathbb{K}$-modules:
$$0 \rightarrow M\xrightarrow{i} S \xrightarrow{p} R\rightarrow 0,$$
where $p:S\rightarrow R$ is a $\mathbb{K}$-algebra homomorphism, and $i(M)$ becomes an $S$-module with the module action: $s.i(m)=i(p(s).m)$.    
  \end{Def} 
Now, there are following interpretations of low dimensional Harrison cohomology spaces:
\begin{Thm}\label{CorresHarr}
\begin{enumerate}
\item The $\mathbb{K}$-module $H^1_{Harr}(R,M)\cong Der_{\mathbb{K}}(R,M)$.
\item There is a bijective correspondence between cohomology classes in $H^2_{Harr}(R,M)$ and the equivalence classes of extensions of $R$ by the $R$-module $M$.
\end{enumerate}
\end{Thm}

\section{Cohomology of Leibniz algebras}

\begin{Def} A  Leibniz algebra (left Leibniz algebra) is a $\Kmath$-module $L$, equipped with a bracket
 operation, which is $\mathbb{K}$-bilinear and satisfies the Leibniz identity,
 $$
 [x,[y,z]]=[[x,y],z] + [y, [x,z]] \quad for \quad x,y,z \in L.
 $$
\end{Def}
In the presence of antisymmetry of the bracket operation, the Leibniz identity is
equivalent to the Jacobi identity, hence any Lie algebra is a Leibniz
algebra.
\begin{Rem}
For a given Leibniz algebra $L$ as defined above,  the left adjoint operation $[x,-]$ is a derivation on $L$ for any
$x \in L.$  Analogously, Leibniz algebra ( right Leibniz algebra ) can be defined by requiring that the right adjoint map $[-,x]$ is  a derivation for any $x \in L.$ In that case, the Leibniz identity appeared in the above definition would be of the form: $$[x,[y,z]]= [[x,y],z]-[[x,z],y] ~~\mbox{for}~x,~y,~z \in L.$$
 Sometime, in the literature ( e.g., \cite{Lod1, Lod2, LP}) right Leibniz algebra has been considered.  All our Leibniz algebras in the present discussion will be left Leibniz algebras unless otherwise stated.
\end{Rem}
\begin{Exm}\label{Lex1}
Let $(\mathfrak{g},[-,-])$ be a Lie algebra and  $V$ be a $\mathfrak{g}$-module with the action $~(x,v)\mapsto x.v$. Take $L=\mathfrak{g}\oplus V $ with the bracket given by
$$[(x,v),(y,w)]_L=([x,y],x.w)~~\mbox{for}~~(x,v), (y,w) \in L .$$
Then $(L,[-,-]_L)$ is a Leibniz algebra. This bracket is called {\it hemisemidirect} product in \cite{KW}.
\end{Exm}

This shows that one can associate a Leibniz algebra $L$ to a smooth manifold $M$. If $\mathfrak{g}$ is the Lie algebra of Vector fields $\chi(M)$ over a smooth manifold $M$, then $C^{\infty}(M)$ is a $\chi(M)$-module with the left action $\chi(M) \times C^{\infty}(M) \rightarrow C^{\infty}(M):~(X,f)\mapsto X(f)$. It follows that  $L= \chi(M)\oplus C^{\infty}(M)$ is a Leibniz algebra with the bracket given by
$$[(X,f),(Y,g)]=([X,Y],X(f)).$$
 There are various other sources to generate more examples of  Leibniz algebras which may not be a Lie algebra.
\begin{Exm}
Let $A$ be an associative $\Kmath$-algebra equipped with a $\Kmath$-linear map $D: A\rightarrow A$, such that $D^2=D.$ 
Define a bilinear map $[-,-]: A\otimes A \rightarrow A$ by 
$$[x,y]:=(Dx)y-y(Dx)~~~ \mbox{for all}~ a,b\in A. $$
Then $(A,[-,-])$ is a left Leibniz algebra. In general, $A$ with the above bracket is not a Lie algebra unless $D=id.$
\end{Exm}

\begin{Exm}
Let $(L,d)$ be a differential Lie algebra with the Lie bracket $[-,-]$. Then $L$ is a Leibniz algebra with the bracket operation $[x,y]_d:= [dx,y]$. The new bracket on $L$ is  called the derived bracket(see \cite{DerB}).
\end{Exm}
One can find more examples appeared in  \cite{Lod1, Lod2, LP} , \cite {HM02, JMF09}.

\begin{Def}
Suppose  $L$ and $L^{\prime}$ are Leibniz algebras, a linear map $\phi: L \rightarrow L^{\prime}$ is called a homomorphism of Leibniz algebras if it preserves the Leibniz bracket, i.e.
$$\phi([x,y]_L)= [\phi(x), \phi (y)]_{L^{\prime}} ~~\mbox{for all}~x,y \in L. $$
\end{Def}

Leibniz algebras with Leibniz algebra homomorphisms form a category of Leibniz algebras, which contains the category of Lie algebras as a full subcategory.

Let $L$ be a Leibniz algebra and $M$ be a representation of $L$. By
definition (\cite{LP,JMF09}), $M$ is a $\Kmath$-module equipped with two actions (left and
right) of $L$,
$$
[-,-]:L \times M \rightarrow M \qquad \text{and} ~~~[-,-]:M \times L
\rightarrow M
$$
such that for $m \in M$ and $x,y \in L$ following hold true,
\begin{equation*}
\begin{split}
[m,[x,y]] =~ & [[m,x],y] + [x,[m,y]]\\
[x,[m,y]]=  ~& [[x,m],y] + [m,[x,y]]\\
[x,[y,m]]= ~ &[[x,y],m] + [y,[x,m]].
\end{split}
\end{equation*}
Define $CL^n(L;M):=\text{Hom}_{\Kmath}(L^{\otimes n},M)$, $n \geq 0$.
Let
$$
\delta^n: CL^n(L;M) \rightarrow CL^{n+1}(L;M)
$$
be a $\Kmath$-homomorphism defined by
\begin{multline}\label{Leibniz coboundary}
(\delta^n \psi)(X_1,X_2, \cdots , X_{n+1}) \\=(-1)^{n+1}\Big(\sum_{i=1}^{n} (-1)^{i-1} [X_i,\psi(X_1, \cdots, \hat{X_i}, \cdots, X_{n+1})]
 +(-1)^{n+1} [\psi(X_1, \cdots, X_{n}), X_{n+1}]  \\+
 \sum_{1\leqslant i < j \leqslant n+1} (-1)^{i}
 \psi(X_1, \cdots, \hat{X_i},\cdots, X_{j-1},[X_i,X_j],X_{j+1}\cdots, X_{n+1})\Big).
\end{multline}

Then $(CL^*(L;M),\delta)$ is a cochain complex, whose cohomology is denoted by $HL^*(L; M)$, is
called the cohomology of the Leibniz algebra $L$ with coefficients in
the representation $M$ (see \cite{LP, JMF09}). In particular, $L$ is a representation of
itself with the obvious actions given by the Leibniz algebra bracket of $L$.


\end{document}